\documentclass{amsart}
\usepackage{medstar}
\usepackage{tikz-cd}
\usepackage[unicode]{hyperref}
\usepackage{graphicx,color}
\usepackage{amssymb}
\usepackage{mathrsfs}
\usepackage{enumitem}
\usetikzlibrary{positioning}
\usepackage{microtype}

\newcommand{\zetavnbhd}{v} 
\newcommand{\zetawnbhdone}{w} 

\newtheorem{thmx}{Theorem}

\usepackage{macros_nerve}
\usepackage{sync_papers}

\title{Modular operads and the nerve theorem}

\author[P. Hackney]{Philip Hackney}
\address{Department of Mathematics\\ University of Louisiana at Lafayette\\ Lafayette, LA 70504-3568 USA}
\email{philip@phck.net} 
\urladdr{http://phck.net}
\thanks{The first author acknowledges the support of Australian Research Council Discovery Project grant DP160101519.}

\author[M. Robertson]{Marcy Robertson}
\address{School of Mathematics and Statistics \\ The University of Melbourne \\ Melbourne, Victoria, Australia}\email{marcy.robertson@unimelb.edu.au}

\author[D. Yau]{Donald Yau}
\address{Department of Mathematics\\
The Ohio State University at Newark\\
Newark, OH \\ USA}
\email{dyau@math.ohio-state.edu}

\date{April 26, 2020}

\begin{document}

\begin{abstract}
We describe a category of undirected graphs which comes equipped with a faithful functor into the category of (colored) modular operads.
The associated singular functor from modular operads to presheaves is fully faithful, and its essential image can be classified by a Segal condition.
This theorem can be used to recover a related statement, due to Andr\'e Joyal and Joachim Kock, concerning a larger category of undirected graphs whose functor to modular operads is not just faithful but also full.
\end{abstract}

\maketitle

The inclusion of the simplex category $\Delta$ into $\cat$,  the category of small categories, induces a fully faithful functor from $\cat$ into the category $\Set^{\Delta^{\oprm}}$ of presheaves, via the assignment $\mathbf{C} \mapsto \mathrm{Fun}(-, \mathbf{C})$.
It is classical that the essential image of this functor consists of those presheaves $X$ which satisfy a \emph{Segal condition}; that is, for every $n\geq 2$ the set $X_n$ can be described as an iterated pullback
\[
	X_n \cong \underbrace{X_1 \times_{X_0} X_1 \times_{X_0} \dots \times_{X_0} X_1}_n.
\]
To goal of this paper is to extend this story to the setting of modular operads.

A modular operad \cite{MR1601666} is an algebraic structure consisting of a sequence of $\Sigma_n$-sets $P(n)$, indexed on nonnegative integers $n$, together with 
\begin{itemize}
	\item `composition operations' $P(n) \times P(m) \to P(n+m-2)$, one for each pair of integers $(i,j) \in [1,n] \times [1,m]$ and
	\item `contraction operations' $P(n) \to P(n-2)$, one for each pair of integers $(i,j)$ with $0\leq i < j \leq n$.
\end{itemize}
This paper, along with its companion \cite{modular_paper_one}, center around a new category of graphs that permit a Segalic approach to the study of modular operads.
This category is a refined version (see Remark~\ref{refinement remark}) of a category of graphs studied by Andr\'e Joyal and Joachim Kock in \cite{JOYAL2011105}.

Our modular graphical category, called $\graphicalcat$, was developed in the companion paper \cite{modular_paper_one}.
The objects of this category are undirected, connected graphs with loose ends, while morphisms are given by `blowing up' vertices of the source into subgraphs of the target in a way that reflects iterated operations in a modular operad.
Regarding a graph as a \emph{colored} modular operad generated by its vertices, one should have a (faithful) functor $\graphicalcat \to \csm$ into the category of colored modular operads.
Our main theorem, which reappears later as Theorem~\ref{nerve_theorem}, is that (colored) modular operads can be characterized as certain objects in the category $\pregraphcat = \Set^{\graphicalcat^{\oprm}}$ of $\graphicalcat$-presheaves.

\begin{thmx}\label{nerve theorem one}
The functor $\graphicalcat \to \csm$ induces a fully-faithful functor $N : \csm \to \pregraphcat$.
The essential image of $N$ consists precisely of those presheaves which satisfy a strict Segal condition.
\end{thmx}

Part of the content of this theorem is the description of the functor from graphs to modular operads.
In this paper, the color set of a modular operad is actually an \emph{involutive} color set, where color matching for composition and contraction operations are governed by the involution (similar to the situation for cyclic multicategories in \cite{MR3189430}).
For example, given a graph $G$, the associated modular operad has color set of cardinality $2|E(G)|$, with one color for each possible orientation on each edge.
If $e$ is an edge of $G$ joining two vertices $v$ and $w$, the generating operations $v$ and $w$ will be tagged with opposing orientations of the edge $e$, and so can be formally composed.

Modular operads as we define them here were first introduced (for the ground category $\Set$) in \cite{JOYAL2011105}, where they are called `compact symmetric multicategories.'
These were further studied in the thesis of Sophie Raynor \cite{Raynor:Thesis} and in \cite{Raynor:DLCSM}.
Most geometric examples of colored modular operads in the literature have been in the setting where the involution on color sets is trivial, as in \cite{Giansiracusa:MSMO}, \cite{HarrelsonVoronovZuniga:OCMSRAS}, and \cite{KaufmannWard:FC}.
A notable exception is \cite{Petersen:OSAGC}, which had a class of examples which were colored by involutive groupoids, rather than involutive sets.
On the other hand, Drummond-Cole and the first author studied colored cyclic operads with involutive set of colors in \cite{DrummondColeHackney:DKHCO}.
Working with involutive color sets had distinct homotopical advantages in that work, which were already clear in \cite[2.11]{doi:10.1112/blms.12232}.
But it had a further advantage: colored cyclic operads (in the sense of \cite{hry_cyclic}), colored operads, and colored dioperads can all be considered as special types of colored cyclic operads when allowing for involutive color sets.
Likewise, our more general notion of modular operad that we consider in this paper allows one to regard wheeled properads as a special case.

The category $\graphicalcat$ from the above is a subcategory of $\csm$, but it is not a full subcategory.
Instead, it is generated by morphisms that are local in nature, involving two or fewer vertices.
In the companion paper (see also Remark~\ref{refinement remark}), this restriction is used to show that $\graphicalcat$ admits a generalized Reedy structure \cite{bm_reedy} (allowing us to use the Reedy model structure on categories of diagrams), which may not be true for the full subcategory $\jkgraphcat$ of $\csm$ spanned by the graphs.
The second theorem of our paper (appearing later as Theorem~\ref{jk nerve theorem}), is the following.

\begin{thmx}[Joyal--Kock 2011]\label{nerve theorem two}
The full subcategory inclusion $\jkgraphcat \to \csm$ induces a fully-faithful functor $N_{JK} : \csm \to \prejkgraphcat$.
The essential image of $N_{JK}$ consists precisely of those presheaves which satisfy a strict Segal condition.
\end{thmx}

This theorem was announced in \cite{JOYAL2011105}, and in Section~\ref{section: JK Nerve} we show how this follows from Theorem~\ref{nerve theorem one}.
This is the first publicly available proof of Theorem~\ref{nerve theorem two}. 
Our proof does not use the techniques proposed by Joyal and Kock.

\subsection*{Related work}
The topic of nerve theorems has a rich literature (that we cannot hope to cover adequately), including a general machine \cite{bmw,weber} that one can use to prove nerve theorems.
This was used by Weber \cite{weber} to prove a nerve theorem for operads involving the dendroidal category $\Omega$ from \cite{mw} (see also the later account \cite[Theorem 2.5.4]{kocktrees}).
This is also the approach towards Theorem~\ref{nerve theorem two} that was indicated in \cite{JOYAL2011105}.
In her thesis \cite{Raynor:Thesis}, Sophie Raynor proved two variations of Theorem~\ref{nerve theorem two} along these lines: one dealt with non-unital modular operads, while the unital version used an alternative category of graphs (see also \cite{Raynor:DLCSM}).
In contrast, Theorem~\ref{nerve theorem one} does not fit into the framework of \cite{bmw}, as $\graphicalcat$ is not a \emph{full} subcategory of $\csm$.
Instead, the situation is more akin to the approach to the dendroidal nerve theorem found in the work of Cisinski, Moerdijk, and Weiss (see, for instance, \cite[Corollary 2.6]{cm-ds}).

\subsection*{Further directions}
In \cite{modular_paper_one} we explained the notion of (inner and outer) coface maps of $\graphicalcat$.
Given a coface map $\delta$ with codomain $G$, one can define the horn $\Lambda^\delta[G]$ which is a subobject of the representable object $\rgc[G]$.
A strict inner Kan presheaf $X$ is a presheaf such that every diagram with $\delta$ an inner coface map
\[ \begin{tikzcd}[column sep=small, row sep=small]
\Lambda^\delta[G] \rar \dar & X \\
\rgc[G] \arrow[ur, dotted, "\exists !" swap]
\end{tikzcd} \]
admits a unique filler.
One could ask if the presheaves of Theorem~\ref{nerve theorem one} coincide with the strict inner Kan presheaves, in analogy with the situation for categories, operads \cite{mw2}, properads \cite{hrybook}, and so on. 
See also Remark~\ref{remark about generators and relations definition}.

\subsection*{Outline}
We begin the paper by recalling, in Section~\ref{section paper2 background}, essential information about $\graphicalcat$ from the companion paper \cite{modular_paper_one}.
Section~\ref{section csm} deals with modular operads, and is split into two subsections, the first of which gives a monadic definition of modular operad valid in any closed symmetric monoidal category $\groundcat$.
As we are working with modular operads with involutive color sets, this definition is technically new, but we regard this section as background.
The heart of the paper begins in \S\ref{sections: graphical csm}, where we construct, for each graph $G$, a modular operad $\modopgenned{G}$.
This is part of a functor $J : \graphicalcat \to \csm$, and in Section~\ref{section: nerve theorem} we use this functor to prove Theorem~\ref{nerve theorem one}.
The most delicate part is found in \S\ref{section mod op segal presheaf}, where we associate a modular operad to any Segal $\graphicalcat$-presheaf.
The final section indicates how to recover Theorem~\ref{nerve theorem two} from Theorem~\ref{nerve theorem one}.

\subsection*{Acknowledgments} 
This paper owes a lot to discussions several years ago with both Andr\'e Joyal and Joachim Kock.
We are also grateful to Sophie Raynor for explaining her thesis work to us.
Finally, we'd like to thank various members of the Centre of Australian Category Theory for their questions and suggestions as this project developed.

\subsection*{Notation}
Let $\mathbf{C}$ be a category.
If $x$ and $y$ are objects of $\mathbf{C}$, we will write either $\hom(x,y)$ or $\mathbf{C}(x,y)$ for the set of morphisms from $x$ to $y$.
We will write $\widehat{\mathbf{C}}$ for the category of $\mathbf{C}$-presheaves, that is, contravariant functors from $\mathbf{C}$ to the category $\Set$ of sets.

\section{Background on the graphical category \texorpdfstring{$\graphicalcat$}{U}}\label{section paper2 background}

All material from this section appears in some form in the companion paper \cite{modular_paper_one}, where proofs and further details may be found.
Here we've only included the essential topics needed to understand what follows.

\begin{figure}[htb]
\includegraphics[scale=.5]{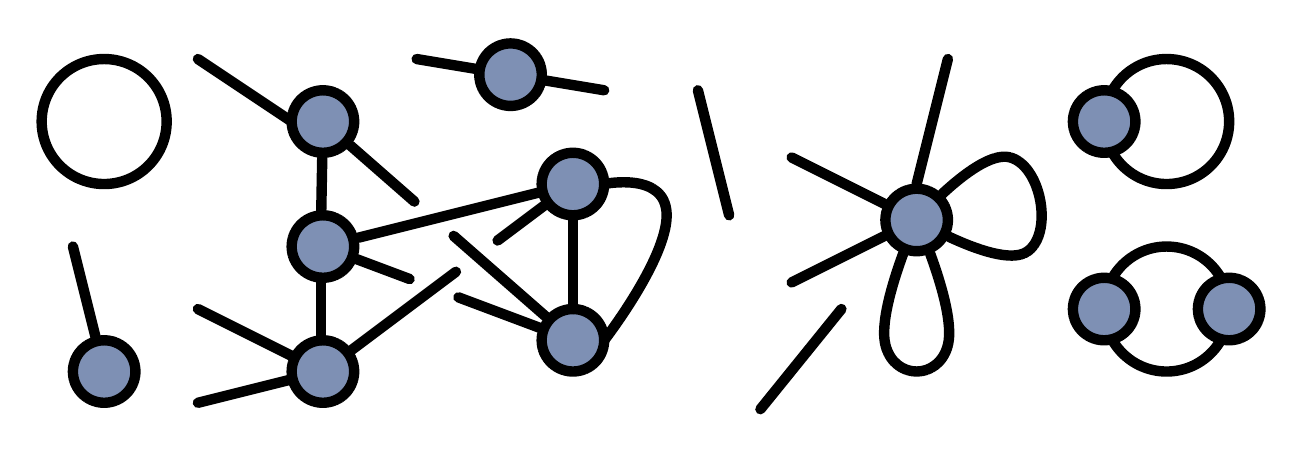}
\caption{A typical graph with loose ends}\label{figure graph picture}
\end{figure}

Graphs in this paper are undirected and are allowed to have `loose ends'; that is, it is not necessary for both ends (or either end) of an edge to touch a vertex.
See Figure~\ref{figure graph picture} for a picture of one such graph.
A model to keep in mind (compare \cite[\S 2]{MR1113284}) is that a graph can be taken to be a pair $(X,V)$ where $X$ is a space, $V$ is a finite set of points of $X$, and $X\setminus V$ is a one-manifold (without boundary) having only a finite set of connected components.
Components of $X\setminus V$ are the edges of the graph, and elements of $V$ are the vertices.
Thus we may have loops divorced from any vertex (those components of $X\setminus V$ homeomorphic to $S^1$), edges loose at one end (those with one missing limit point in $X$), and free floating edges (components of $X$ homeomorphic to $(0,1)$ which contain no vertices).

Such pictures lead to the following definition.
The involutive set $A$ is the set of \emph{arcs}, which are edges together with an orientation, and the involution $i$ swaps orientation.
The partially-defined function $t: A \nrightarrow V$ takes an arc to the vertex it points towards.

\begin{definition}\label{definition graphs}
A \emph{graph} $G$ consists of 
\begin{enumerate}[label=(\alph*), ref=\alph*]
\item 
a diagram of finite sets\footnote{To ensure that we have a \emph{set} of graphs, insist that all of the sets $A, D, V$ are taken to be subsets of some fixed infinite set.} 
\[
\begin{tikzcd}
	A \arrow[loop left, "i"] & \lar[swap]{s} D \rar{t} & V,
\end{tikzcd}
\]
and
\item
a subset $\eth(G) \subseteq A \setminus sD$ called the \emph{boundary} of $G$\label{boundary set}
\end{enumerate}
so that
\begin{enumerate}[label=(\Alph*), ref=\Alph*]
\item $i$ is a fixedpoint-free involution,
\item $s$ is a monomorphism,
\item $isD \setminus sD \subseteq \eth(G)$, and
\item $\eth(G) \setminus isD$ is an $i$-closed subset of $A$.
\end{enumerate}
We will nearly always consider $D$ as a subset of $A$, and suppress the natural inclusion function $s : D \subseteq A$ from the notation.
A graph will be called \emph{safe} if $\eth(G) = A \setminus sD$, while if the containment from \eqref{boundary set} is strict then the graph will be called \emph{unsafe}.
\end{definition}

This definition is a modification of that in \cite{JOYAL2011105} in that it has a specified notion of boundary.
There are several other combinatorial definitions of graphs \cite{batanin-berger,MR1113284,yj15}
all of which are equivalent (Proposition 15.2, Proposition 15.6, and Proposition 15.8 of \cite{batanin-berger}) to this one.

\begin{example}[Exceptional edge and nodeless loop]\label{examples edge and loop}
If $Z$ is a set, write $2Z$\index{$2Z$} for the set
\[
	\{ z, z^\dagger \mid z\in Z\} \cong Z \amalg Z
\]
together with the evident involution.
We consider $Z$ as a subset of $2Z$, and write $Z^\dagger$\index{$Z^\dagger$} for its complement.
\begin{itemize}
\item 
The \emph{exceptional edge}, $\exceptionaledge$, is the safe graph with $A = 2\{*\} = \eth(\exceptionaledge)$ and $V=D=\varnothing$.
As this graph is so important, we give special names to its arcs and write $A = \eearcs$.
\item 
A variation is to take $A = 2\{*\}$, $V=D=\varnothing$, and have an empty boundary.
We call graphs isomorphic to this one \emph{nodeless loops}.
\end{itemize}
\end{example}

Recall that the neighborhood of a vertex $v\in V$ is defined to be $ \nbhd(v) = t^{-1}(v) \subseteq D$.
The \emph{valence} of a vertex $v$ is just the cardinality of the set $\nbhd(v)$.

Many other examples of graphs are given in the companion paper \cite{modular_paper_one}.
For instance, if $G$ is connected and every vertex is bivalent, then $G$ is either a linear graph or a cycle.

\begin{definition}[Stars]\label{definition star}
For $n\geq 0$, the $n$-star $\medstar_n$\index{$\medstar_n$} has $V$ a one-point set, $D = \{ 1, \dots, n\}$, and $A = 2D$ (where $2(-)$ is as in Example~\ref{examples edge and loop}).
The function $s : D \to A = 2D$ is just the subset inclusion.
More generally, if $S$ is any set we define $\medstar_{S}$ to be the (connected) graph with a single vertex so that $A = 2S$, $D=S^\dagger \subseteq 2S$, and $\eth(\medstar_{S}) = S \subseteq 2S$.
There are also variations of stars built from a fixed (connected) graph $G$.
\begin{itemize}
\item 
Let $\medstar_G$ be the one-vertex graph with $A = 2\eth(G)$ and $D= \eth(G)^\dagger$.
Notice that we must have $\eth(\medstar_G) = A \setminus D = \eth(G)$ and that the neighborhood of the unique vertex is $D = \eth(G)^\dagger$.
In other words, $\medstar_G = \medstar_{\eth(G)}$.
\item Suppose that $v$ is a vertex of $G$ and let $\nbhd(v)$ be its neighborhood in $G$.
We let $\medstar_v$ denote the graph with $V = \{ v \}$, $D = \nbhd(v)$, and $A = 2\nbhd(v)$.
The boundary of $\medstar_v$ is $\nbhd(v)^\dagger \subseteq 2\nbhd(v)$.
There is a canonical embedding
\[
	\iota_v : \medstar_v \to G
\]
coming from the natural inclusions $\{ v \} \subseteq V(G)$ and $\nbhd(v) \subseteq D(G)$.
\end{itemize}
\end{definition}

Let us now recall several varieties of morphisms from \cite{modular_paper_one} and \cite{JOYAL2011105}.

\begin{definition}[Natural transformations of graphs]\label{def nat trans}
Let $\mathscr{I}$ denote the category with three objects and three generating arrows, of shape $\begin{tikzcd}[column sep=small] \bullet \arrow[loop left] & \lar \bullet \rar & \bullet. \end{tikzcd}$
Part of the data of each graph $G$ is a functor from $\mathscr{I}$ into finite sets where the leftward arrow is sent to a monomorphism and the generating endomorphism is sent to a free involution.
\begin{itemize}
	\item A graph is \emph{connected} if this functor is connected as an object in $\finset^{\mathscr{I}}$ (that is, if it is nonempty and cannot be written as a nontrivial coproduct in this category).
	\item If $G$ and $G'$ are graphs, then a natural transformation $G\to G'$ is said to be \emph{\'etale} if
	\begin{enumerate}
	 	\item the right-hand square of
\begin{equation*}
\begin{tikzcd}
	A \arrow[loop left, "i"]\dar & \lar[swap]{s} D\dar \rar{t} & V\dar \\
	A' \arrow[loop left, "i'"] & \lar[swap]{s'} D' \rar{t'} & V'
\end{tikzcd}
\end{equation*}
is a pullback, and \label{etale first condition}
\item the set $A\setminus (\eth(G) \amalg D)$ maps into $A' \setminus (\eth(G') \amalg D')$. \label{etale second condition}
	 \end{enumerate} 
	\item If $G$ and $G'$ are connected graphs, then an \'etale map is called an \emph{embedding} if $V \to V'$ is a monomorphism.
	\item The set $\bigembeddings(G)$ consists of all embeddings with codomain $G$.
	The set $\embeddings(G)$ is the quotient of $\bigembeddings(G)$ by the relation that $f \sim h$ whenever there is an isomorphism $z$ with $f=hz$.
\end{itemize}
\end{definition}

Note that \eqref{etale second condition} is automatically satisfied when $G$ is safe.
The original definition of \'etale, from \cite{JOYAL2011105} only had condition \eqref{etale first condition} as all graphs were implicitly regarded as safe.

In order to state our definition of graphical map from \cite{modular_paper_one}, we need two supplementary definitions.
Both of these are initially functions on $\bigembeddings(G)$, but as we saw in the companion paper these descend to $\embeddings(G)$.

\begin{definition}[Invariants of embeddings]
Suppose that $G$ and $G'$ are two (potentially unsafe) graphs.
\begin{itemize}
	\item Given any \'etale map $f: G'\to G$, there is a corresponding element \[ \sum_{v\in V'} f(v) \in \mathbb{N}V \] in the free commutative monoid on $V$.
The \emph{vertex sum}, denoted $\varsigma : \embeddings(G) \to \mathbb{N}V$, is the function that takes $[f: G' \to G]$ to $\sum_{v\in V'} f(v)$.
As we restrict to embeddings, this function factors through the power set $\wp(V) \subseteq \mathbb{N}V$.
	\item The restriction of any embedding $f : G' \to G$ to the boundary $\eth(G')$ is a monomorphism.
	We write $\eth: \embeddings(G) \to \wp(A(G))$ for the function which takes $[f : G'\to G]$ to $f(\eth(G')) \subseteq A(G)$.
\end{itemize}
\end{definition}

\begin{definition}[Graphical category]\label{graphical category definition}
The \emph{graphical category} $\graphicalcat$ has objects the safe, connected graphs.
A morphism $\varphi: G \to G'$ (where $G$ and $G'$ are safe) consists of the following data:
	\begin{itemize}
		\item A map of involutive sets $\varphi_0 : A \to A'$ 
		\item A function $\varphi_1 : V \to \embeddings(G')$
	\end{itemize}
These data should satisfy two conditions.
\begin{enumerate}[label={({\roman*})},ref={\roman*}]
	\item The inequality $\sum_{v\in V} \varsigma(\varphi_1(v)) \leq \sum_{w\in V'} w$ holds in $\mathbb{N}V'$. \label{graphical map defn: no double vertex covering}
	\item For each $v$, we have a (necessarily unique) bijection making the diagram
	\[
		\begin{tikzcd}
			\nbhd(v) \rar{i} \dar[dashed, "\cong"]  & A \dar{\varphi_0} \\
			\eth(\varphi_1(v)) \rar[hook] & A'
		\end{tikzcd}
	\]
	commute, where the top map $i$ is the restriction of the involution on $A$.\label{graphical map defn: boundary compatibility}
	\item If the boundary of $G$ is empty, then there exists a $v$ so that $\varphi_1(v)$ is not an edge. \label{graphical map defn: collapse condition}
\end{enumerate}
The \emph{extended graphical category} $\egc$ is defined similarly, except the objects are allowed to be arbitrary connected graphs and condition \eqref{graphical map defn: collapse condition} for morphisms is replaced by 
\begin{enumerate}[label={({\roman*}')},ref={\roman*'},start=3]
	\item If the boundary of $G$ is empty and $\varphi_1(v)$ is an edge for every $v$, then $G'$ is a nodeless loop. \label{egc defn: collapse condition}
\end{enumerate}
\end{definition}

The composition in $\graphicalcat$ and $\egc$ are given by graph substitution.
Let us recall the idea; a precise definition in our setting appears in Definition~\ref{def segal core and graph sub}.
Suppose that we are given a graph $G$, a collection of graphs $H_v$ indexed by the vertices of $G$, and specified
bijections $i\nbhd(v) \cong \eth(H_v)$.
Then we can form a new graph $G\{H_v\}$ where we replace each vertex $v$ by the graph $H_v$, identifying the edges at the boundary of $H_v$ with the edges incident to the vertex $v$ in $G$.

\begin{definition}[Composition of graphical maps]\label{def graph map composition}
Suppose that $\varphi : G \to G'$ and $\psi : G' \to G''$ are graphical maps.
We will define the composite $\psi \circ \varphi$.
First, we have $(\psi \circ \varphi)_0 = \psi_0 \circ \varphi_0$.
To define $(\psi \circ \varphi)_1(v)$, where $v$ is a vertex of $G$, first let $\varphi_v : K_v \hookrightarrow G'$ be an embedding representing $\varphi_1(v)$.
For a vertex $w$ in $\varphi_v(V(K_v)) \subseteq V(G')$, we can find an embedding $\psi_w : H_w \hookrightarrow G''$ representing $\psi_1(w)$.
It turns out that the $\psi_w$ assemble into a single embedding\footnote{The fact that this is an embedding and not merely \'etale follows from \eqref{graphical map defn: no double vertex covering} of Definition~\ref{graphical category definition}.}
\begin{equation}\label{eq composite embedding}
K_v \{ H_w \} \hookrightarrow G''
\end{equation}
which factors each of the embeddings $\psi_w$.
The function $(\psi \circ \varphi)_1$ sends $v$ to the class of \eqref{eq composite embedding} in $\embeddings(G'')$.
\end{definition}

See \cite{modular_paper_one} for further details.

\begin{remark}
\label{refinement remark}
There is a related notion of morphism of connected graphs in \cite{JOYAL2011105}, but based on \'etale maps between connected safe graphs, rather than embeddings.
Joyal and Kock do not include the conditions \eqref{graphical map defn: no double vertex covering} and \eqref{graphical map defn: collapse condition} of Definition~\ref{graphical category definition} in their definition.
Further, condition \eqref{graphical map defn: boundary compatibility} is modified to reflect that \'etale maps need not be injective on boundaries.
This yields a category of connected safe graphs $\jkgraphcat$, and each graphical map in the sense of Definition~\ref{graphical category definition} is a morphism in $\jkgraphcat$.
The weak factorization system that is meant to exist on the category of Joyal and Kock becomes an orthogonal factorization system on our category,\footnote{Compare with \cite[2.4.14]{Kock:GHP} in the directed setting, which is much simpler as embeddings in that context are monomorphisms.} which is much easier to work with.
Moreover, our category admits a generalized Reedy structure in the sense of \cite{bm_reedy}, allowing us flexibility when considering model structures in the companion paper \cite{modular_paper_one}.
\end{remark}

Embeddings constitute the right class of an orthogonal factorization system on $\graphicalcat$ (resp. on $\egc$).
Morphisms in the left class are called \emph{active maps}.

\begin{definition}[Active maps]\label{def active and star active}
A morphism $\varphi: G \to G'$ is called \emph{active} if $\varphi_0: A \to A'$ induces a bijection $\eth(G) \to \eth(G')$.
\begin{itemize}
	\item If $G$ is a graph, there is a canonical active map $\medstar_G \to G$ (see Definition~\ref{definition star}) which sends the unique vertex of $\medstar_G$ to $[\id_G : G \to G]$ and on arcs gives the identity on $\eth(\medstar_G) = \eth(G)$.
	\item More generally, if $G$ is a graph, $S$ is a set, and $\xi : S \to \eth(G)$ is a function, then there is an associated active map $\medstar_S \to G$ whose map on arcs restricts to $\xi : S \to \eth(G) \subseteq A(G)$.
\end{itemize}
\end{definition}

Before Definition~\ref{def graph map composition}, we mentioned the idea of graph substitution. 
In Construction~\ref{construction assembly}, it will be helpful to have a concrete model on hand.
Further, the notion of the Segal core of a graph is essential throughout this paper.
As these concepts are closely related, we combine them into a single definition.
Recall that if $G$ is a graph then the representable presheaf $\rgc[G]$ is the contravariant functor from $\graphicalcat$ to $\Set$ with $\rgc[G]_H = \graphicalcat(H,G)$.

\begin{definition}[Graph substitution and Segal cores]\label{def segal core and graph sub}
Suppose that $G$ is a connected graph containing at least one vertex, and let $E_i$ be its set of internal edges.
For each internal edge $e\in E_i$, choose an ordering $e = [x_{e}^1, x_{e}^2]$ for the two-element equivalence class of arcs comprising $e$.
The underlying functor of $G$ (in the diagram category $\finset^{\mathscr{I}}$) may be regarded as a coequalizer
\begin{equation} \label{eq coequalizer}
\begin{tikzcd}
\coprod\limits_{e \in E_i} \exceptionaledge \rar[shift left, "\outeredge"] \rar[shift right, "\inneredge" swap] & \coprod\limits_{v\in V} \medstar_v \rar & G,
\end{tikzcd} \end{equation}
where the map on the right is $\coprod_v \iota_v$.
Explicitly, we have
\begin{itemize}
	\item $\outeredge$ is the coproduct of maps $\outeredge_e : \exceptionaledge \to \medstar_{tx_e^1}$ with $\outeredge_e(\edgemajor) = (x_e^1)^\dagger \in \eth(\medstar_{tx_e^1})$ and $\outeredge_e(\edgeminor) = x_e^1 \in D(\medstar_{tx_e^1})$;
	\item $\inneredge$ is the coproduct of maps $\inneredge_e :  \exceptionaledge \to \medstar_{tx_e^2}$ with $\inneredge_e(\edgeminor) = (x_e^2)^\dagger \in \eth(\medstar_{tx_e^2})$ and $\inneredge_e(\edgemajor) = x_e^2 \in D(\medstar_{tx_e^2})$.
\end{itemize}
\begin{itemize}[label=$\blacktriangleright$, leftmargin=*]
\item 
We first describe \emph{graph substitution}.
Suppose we are given graphs $H_v$ and isomorphisms $m_v$ from $i(\nbhd(v)) \subseteq A(G)$ to $\eth(H_v)$.
We then have induced maps $\tilde \outeredge$ and $\tilde \inneredge$, where
\begin{itemize}
	\item $\tilde \outeredge$ is the coproduct of maps $\tilde \outeredge_e : \exceptionaledge \to H_{tx_e^1}$ with $\tilde \outeredge_e(\edgemajor) = m_{tx_e^1}(ix_e^1) \in \eth(H_{tx_e^1})$, 
	\item $\tilde \inneredge$ is the coproduct of maps $\tilde \inneredge_e : \exceptionaledge \to H_{tx_e^2}$ with 
	 $\tilde \inneredge_e(\edgeminor) = m_{tx_e^2}(ix_e^2) \in \eth(H_{tx_e^2})$.
\end{itemize}
We can then form the coequalizer 
\[ \begin{tikzcd}
\coprod\limits_{e \in E_i} \exceptionaledge \rar[shift left, "\tilde \outeredge"] \rar[shift right, "\tilde \inneredge" swap] & \coprod\limits_{v\in V} H_v \rar["\pi"] & K.
\end{tikzcd} \]
There is an induced monomorphism (see \cite{modular_paper_one}) $\eth(G) \to A(K) \setminus D(K)$ and we declare $\eth(K)$ to be the image of this function.
We write $G\{H_v\}$ for this graph, called graph substitution of $H_v$ into $G$.
\item
We likewise can form corresponding coequalizer to \eqref{eq coequalizer} in $\pregraphcat$,
\[ \begin{tikzcd}
\coprod\limits_{e \in E_i} \rgc[\exceptionaledge] \rar[shift left, "\outeredge"] \rar[shift right, "\inneredge" swap] & \coprod\limits_{v\in V} \rgc[\medstar_v] \rar & \segalcore[G],
\end{tikzcd} \]
and we call the target the \emph{Segal core} of $G$.
It comes with a map $\segalcore[G] \to \rgc[G]$ induced by $\coprod \iota_v : \coprod_{v\in V} \rgc[\medstar_v] \to \rgc[G]$.
In the case when $G = \exceptionaledge$, we declare the map $\segalcore[G] \to \rgc[G]$ to be the identity map on $\rgc[G]$.
\end{itemize}
\end{definition}

We return to Segal core definitions in a different context in Notation~\ref{notation segal core subscript}.

\section{Modular operads}\label{section csm}

In this section, we define (colored) modular operads in a closed monoidal category (\S\ref{subsection CSM}) and fabricate a class of examples coming from graphs (\S\ref{sections: graphical csm}).
Our modular operads come equipped with an involution on color sets, and are an enriched version of the \emph{compact symmetric multicategories} introduced in \cite{JOYAL2011105}.
All of the examples in \S\ref{sections: graphical csm} in fact come equipped with \emph{free} involutions on the sets of colors.

\begin{remark}
At first glance it may appear that \S\ref{subsection CSM} depends on our particular choice of graph formalism (Definition~\ref{definition graphs}).
In fact, our constructions are relatively formalism agnostic, as long as we can get a handle on what the set of arcs (and the involution on that set) of a graph should be.
For example, if one chooses to use Yau--Johnson graphs as in \cite[\S 1.2]{yj15}, then the set of arcs $A$ may be identified with $\flag(G)\amalg\legs_{\mathrm{o}}(G)$.
The involution on $A$ is uniquely specified so that it
\begin{itemize}
	\item acts on this added $\legs_{\mathrm{o}}(G)$ component by including into $\flag(G)$,
	\item acts on $\flag(G) \setminus \legs(G)$ via $\iota_G$,
	\item acts on $\legs_{\mathrm{e}}(G) \subseteq \flag(G)$ by $\pi_G$.
\end{itemize}
We consider the subset $\legs_{\mathrm{o}}(G) \subseteq \flag(G)$ as being part of neighborhoods for some vertices, while the new summand $\legs_{\mathrm{o}}(G)$ constitutes part of the boundary $\eth(G)$.
Specifically, $\eth(G)$ is the sum of this added $\legs_{\mathrm{o}}(G)$ and $\legs_{\mathrm{e}}(G) \subseteq \flag(G)$.
\end{remark}

Throughout this section `graph' will mean `connected graph' unless otherwise indicated.
We emphasize that we are generally including nodeless loops as well, which is important in order to avoid the issue mentioned in Remark~\ref{remark about joyal-kock}.

\subsection{Monads governing modular operads}\label{subsection CSM}

Let us fix a cocomplete, closed, symmetric monoidal category $(\groundcat, \otimes, 1)$. 
In this subsection we give a monadic description of $\fC$-colored modular operads in $\groundcat$, where $\fC$ is an involutive set.
The monad in question is an adaptation of other existing monads for modular operads (\cite[\S 7]{Markl:OP}, \cite[2.17]{MR1601666}, \cite[\S10.1]{batanin-berger}) and generalized operadic structures (\cite[\S6]{batanin-berger}, \cite[10.2,10.3]{yj15}).
It is also closely related to the monad from \cite[\S 5]{JOYAL2011105}; see Remark~\ref{remark about joyal-kock}.
As such, the chief aim of the beginning of this subsection is to fix notation and provide enough background for the remainder of the paper.
In Definition~\ref{definition CSM} we explain how to define morphisms between modular operads with different color sets.

\begin{definition}
\label{preliminary definition bijections over C}
Let $\bijcat_{\fC}$ denote the groupoid with:
\begin{itemize}
	\item objects pairs $(S,\xi)$, where $S$ is a finite set and $\xi : S \to \fC$ is a function, and
	\item morphisms $(S,\xi) \to (S',\xi')$ are bijections $f  : S \to S'$ so that $\xi = \xi' \circ f$.
\end{itemize}
\end{definition}
In particular, $\bijcat_{\{*\}}$ is just the usual category of finite sets and bijections. 
Note that Definition~\ref{preliminary definition bijections over C} ignores the involution present on the set $\fC$.

\begin{remark}
\label{remark on skeletality}
We could instead restrict this definition to the finite sets $\{1, \dots, n\}$.
In this case, a coloring function $\xi$ is the same thing as an ordered list $c_1, \dots, c_n$ of elements of $\fC$.
Suppose $\sigma$ is an automorphism of $\{1,\dots, n\}$, considered as a morphism of $\bijcat_{\fC}$ from $\xi \to \xi \sigma^{-1}$.
Using the identification of $\xi$ with the list $c_1, \dots, c_n$ and likewise for $\xi\sigma^{-1}$, the morphism $\sigma$ goes from $c_1, \dots, c_n$ to $c_{\sigma^{-1}(1)}, \dots, c_{\sigma^{-1}(n)}$.
Thus we can identify $\Sigma_{\fC}$ from \cite[Definition 2.11]{HackneyRobertsonYau:SRLPDBC} with the full subcategory of $\bijcat_{\fC}$ whose objects have the form $(\{1,\dots, n\}, \xi)$ for some $n$.

The functor $\Sigma_{\fC} \to \bijcat_{\fC}$ is, in fact, an equivalence of categories.
Everything we're doing in this section could actually be done `skeletally', that is, by restricting our constructions to $\Sigma_{\fC}$.
This would require us to consider graphs with extra structure, namely orderings on each set $\nbhd(v)$ and on $\eth(G)$.
We've typically taken this approach in earlier work (for example, in \cite{hry_cyclic} which also deals with the undirected context), but will not do so here.
This choice allows us to track certain other papers (e.g., \cite{Doubek:MECOA,JOYAL2011105,Markl:MEONMO}) more closely.
\end{remark}

\begin{notation}
\label{notation inclusion}
If $Z$ is a subset of $\fC$, we will write $\incl : Z \hookrightarrow \fC$ for the inclusion.
\end{notation}

We now define certain graph groupoids.

\begin{definition}[Groupoids of colored graphs]
\label{definition graph groupoids}
Let $\fC$ be a set equipped with an involution $c\mapsto c^\dagger$.
\begin{itemize}
\item A $\fC$-colored graph is a graph $G$ together with an involutive map $\zeta: A \to \fC$.
\item Let $\conngraph_{\fC}$ be the groupoid whose objects are $\fC$-colored graphs and whose isomorphisms $(G,\zeta) \to (G',\zeta')$ are graph isomorphisms $z : G\to G'$ so that 
\[ \tcdtriangle{A}{\zeta}{z}{A'}{\zeta'}{\fC} \]
commutes.
\item There is a functor, which we call $\eth_{\fC}$,
\[ \begin{tikzcd}[row sep=tiny]
\conngraph_{\fC} \rar["\eth_{\fC}"] & \bijcat_{\fC} \\
(G,\zeta) \rar[mapsto] & (\eth(G), \zeta|_{\eth(G)}).
\end{tikzcd} \]
\item If $(S,\xi) \in \bijcat_{\fC}$, let $\conngraph_{\fC}(S,\xi)$ denote the category $(S,\xi) \downarrow \eth_{\fC}$.
\end{itemize}
\end{definition}

Let's unravel this last definition.
An object of $\conngraph_{\fC}(S,\xi)$ consists of a triple $(f,G,\zeta)$ where $(G,\zeta)$ is a $\fC$-colored graph and $f: S \to \eth(G)$ is a bijection so that $\zeta|_{\eth(G)} \circ f = \xi$.
An isomorphism $(f,G,\zeta) \to (f',G',\zeta')$ is a graph isomorphism $z: G \to G'$ so that the diagram
\[ \begin{tikzcd}[column sep=small]
& S \arrow[dl, "f" swap] \arrow[dr, "f'"] \\
\eth(G) \dar[hook] \arrow[rr] & & \eth(G') \dar[hook] \\
A \arrow[dr,"\zeta" swap] \arrow[rr,"z"] & & A' \arrow[dl,"\zeta'"] \\
& \fC
\end{tikzcd} \]
commutes.

\begin{remark}
The groupoid $\conngraph_{\{*\}}$ is equivalent to $\Iso(\graphicalcat) \amalg C_2$, where $C_2$ is the cyclic group of order 2 (considered as a one-object groupoid).
Indeed, the only connected graphs that do not already appear in $\graphicalcat$ are nodeless loops, each of which has a single nontrivial automorphism.
\end{remark}

Notice that if $\ell : (S,\xi) \to (S',\xi')$ is an isomorphism of $\bijcat_{\fC}$, then we have an induced functor in the reverse direction $\conngraph_{\fC}(S,\xi) \leftarrow \conngraph_{\fC}(S',\xi') : \ell^*$ taking $(f,G,\zeta)$ to $(f\circ \ell, G, \zeta)$.
This is of course an isomorphism, and we write $\ell_! : \conngraph_{\fC}(S,\xi) \to \conngraph_{\fC}(S',\xi')$ for the functor sending $(f,G,\zeta)$ to $(f \circ \ell^{-1},G,\zeta)$.
In other words, we are considering $\conngraph_{\fC}(-)$ as a functor from $\bijcat_{\fC}$ to the category of groupoids. 

Before approaching the next definition, we introduce some convenient shorthand which we use for the remainder of this subsection. 
Suppose that $(G,\zeta)$ is a $\fC$-colored graph and $X$ is an object of $\groundcat^{\bijcat_{\fC}}$.
We will write $X(v)$ for the object 
\[ X(v) = X(\dnbhd{v}, \zeta|_{\dnbhd{v}})\] 
in $\groundcat$, suppressing the colored graph $(G,\zeta)$ from the notation.
Likewise, for graph groupoids, we write 
\[
	\conngraph_{\fC}(v) = \conngraph_{\fC}(\dnbhd{v}, \zeta|_{\dnbhd{v}}).
\]

\begin{definition}[Decorations]\label{defn obj of decorations}
Suppose given an object $X\in \groundcat^{\bijcat_{\fC}}$.
\begin{enumerate}
\item Let $(G,\zeta)$ be a $\fC$-colored graph. Define the object 
\[ X[G,\zeta] = \bigotimes_{v\in V} X(v) = \bigotimes_{v\in V} X(\dnbhd{v}, \zeta|_{\dnbhd{v}}) \]
in $\groundcat$.
\item A \emph{decoration of $G$ by $X$} or an \emph{$X$-decoration of $G$} consists of an involutive function $\zeta : A(G) \to \fC$ and an element of $X[G,\zeta]$. 
\item The assignment $(X,(G,\zeta)) \mapsto X[G,\zeta]$ is 
the object part of a bifunctor
\[ \groundcat^{\bijcat_{\fC}} \times \conngraph_{\fC} \to \groundcat. \]
\end{enumerate}
\end{definition}

\begin{construction}[Colored graph substitution]\label{construction assembly}
Suppose that $(G,\zeta)$ is a $\fC$-colored graph.
We describe an associated functor
\[
	\prod_{v\in V} \conngraph_{\fC}(\zetavnbhd) \to \conngraph_{\fC}(\eth(G), \zeta|_{\eth(G)}).
\]
Let $(m_v, H_v, \zeta_v)$ denote an object of $\conngraph_{\fC}(\zetavnbhd) = \conngraph_{\fC}(\dnbhd{v}, \zeta|_{\dnbhd{v}})$, where $m_v : \dnbhd{v} \to \eth(H_v)$ is a bijection satisfying $\zeta|_{\dnbhd{v}} = \zeta_v \circ m_v$.
Then $\prod_v (m_v, H_v, \zeta_v)$ will map to an object of the form 
$(\eth(G) \xrightarrow{\cong} \eth(G\{H_v\}), G\{H_v\}, \tilde \zeta)$.
Here, the graph  substitution $G\{H_v\}$ is defined using the bijections $m_v$.
The coloring function $\tilde \zeta$ is induced from $\zeta$ and the $\zeta_v$.
Specifically, the underlying functor part of the graph substitution is described in Definition~\ref{def segal core and graph sub}.
Since colimits in functor category $\finset^{\mathscr{I}}$ are computed objectwise, we have a coequalizer diagram and an induced map 
\[ \begin{tikzcd}
\coprod\limits_{e \in E_i} A(\exceptionaledge) \arrow[dr, "\zeta"] \rar[shift left, "\tilde \outeredge"] \rar[shift right, "\tilde \inneredge" swap] & \coprod\limits_{v\in V} A(H_v) \rar["\pi"] \dar["\zeta_v"]& A(G\{H_v\}) \arrow[dl, dashed, "\tilde \zeta" swap] \\
& \fC
\end{tikzcd} \]
into $\fC$. The map $\eth(G) \to \eth(G\{H_v\})$ is the canonical identification of $\eth(G)$ with $\eth(G\{H_v\})$.
\end{construction}

Graph substitution induces an endofunctor $\freegraphmonad_{\fC} = \freegraphmonad \colon \groundcat^{\bijcat_{\fC}} \rightarrow \groundcat^{\bijcat_{\fC}}$ where 
\[ (\freegraphmonad X)(S,\xi) = \colim_{(f,G,\zeta) \in \conngraph_{\fC}(S,\xi)} X[G,\zeta].\]
Our next goal is to show that $\freegraphmonad$ can be given the structure of a monad (Proposition~\ref{proposition free graph monad}).
Let us first define $\mu :\freegraphmonad\freegraphmonad \Rightarrow \freegraphmonad$; 
it is sufficient to define, for $(f,G,\zeta) \in \conngraph_{\fC}(S,\xi)$ and $X\in \groundcat^{\bijcat_{\fC}}$ the composites
\[
	\freegraphmonad X[G,\zeta] \to \freegraphmonad\freegraphmonad X (S,\xi) \to \freegraphmonad X (S,\xi).
\]
We have the following equalities 
\begin{align*}
\freegraphmonad X[G,\zeta] =  \bigotimes_{v\in V} \freegraphmonad X(\zetavnbhd) 
&= \bigotimes_{v\in V} \colim_{\substack{(m_v, H_v, \zeta_v) \in \\ \conngraph_{\fC}(\zetavnbhd) }} X[H_v, \zeta_v] \\
&\cong \colim_{\prod_v \conngraph_{\fC}(\zetavnbhd)} \bigotimes_{v\in V} X[H_v, \zeta_v],
\end{align*}
where the isomorphism comes from the fact that $\groundcat$ is closed (so $\otimes$ commutes with colimits).
Further, we have
\[
\bigotimes_{v\in V(G)} X[H_v, \zeta_v] \cong \bigotimes_{w\in V(G\{H_v\})} X(\zetawnbhdone) = X[G\{H_v\}, \tilde \zeta]
\]
where $\tilde \zeta$ is the coloring for $G\{H_v\}$ appearing in Construction~\ref{construction assembly}.
Thus graph substitution provides the first map in the composite 
\begin{equation}\label{eq mu action}
\freegraphmonad X[G,\zeta] \to \colim_{\conngraph_{\fC}(\eth(G), \zeta|_{\eth(G)})} X[K,\zeta'] \xrightarrow{\cong} (\freegraphmonad X)(S,\xi),	
\end{equation}
while the second morphism comes from the functor $\conngraph_{\fC}(\eth(G), \zeta|_{\eth(G)}) \to \conngraph_{\fC}(S,\xi)$ induced by $f : (S,\xi) \to (\eth(G),\zeta|_{\eth(G)})$.

\begin{remark}
\label{remark mu when G lacks vertices}
The above degenerates into something much simpler for $(f,G,\zeta) \in \conngraph_{\fC}(S,\xi)$ when $G$ has no vertices.
In that case, both $\freegraphmonad X[G,\zeta]$ and $X[G,\zeta]$ are the tensor unit.
Further, what would usually be the structural map $X[G\{H_v\}, \tilde \zeta] \to (\freegraphmonad X)(\eth(G), \zeta|_{\eth(G)})$ just becomes a map from $X[G,\zeta]$ at the object $(\id, G, \zeta) \in \conngraph_{\fC}(\eth(G), \zeta|_{\eth(G)})$.
That is, \eqref{eq mu action} factors through this structural map:
\[ \begin{tikzcd} 
& X[G,\zeta] \dar\\
\freegraphmonad X[G,\zeta] \rar \arrow[ur,dashed] &  \colim\limits_{\conngraph_{\fC}(\eth(G), \zeta|_{\eth(G)})} X[K,\zeta'] \rar & (\freegraphmonad X)(S,\xi).
\end{tikzcd} \]
\end{remark}

We now turn to the unit $\eta : \id \Rightarrow \freegraphmonad$.
For this, the following definition is helpful.
\begin{definition}\label{def star coloring}
Let $(S,\xi)$ be an object of $\bijcat_{\fC}$. 
Recall that the graph $\medstar_S$ from Definition~\ref{definition star} has a single vertex, $A = 2S$, and $\eth(\medstar_S) = S$.
There is a unique involutive extension $\xi^\medstar : 2S \to \fC$ of $\xi : S \to \fC$, namely the one with $\xi^\medstar(s) = \xi(s)$ and $\xi^\medstar(s^\dagger) = (\xi(s))^\dagger$.
\end{definition}
If $(S,\xi) \in \bijcat_{\fC}$, then $X[\medstar_S, \xi^\medstar] = X(S,\xi)$.
The map
\[
	\eta_{(S,\xi)} : X(S,\xi) \to \freegraphmonad X (S,\xi)
\]
is defined to be the structural map 
\[ X[\medstar_S, \xi^\medstar] \to \colim_{(f,G,\zeta) \in \conngraph_{\fC}(S,\xi)} X[G,\zeta]\]
associated with the object $(\id_S, \medstar_S, \xi^\medstar) \in \conngraph_{\fC}(S,\xi)$.

\begin{proposition}
\label{proposition free graph monad}
The functor $\freegraphmonad = \freegraphmonad_{\fC} : \groundcat^{\bijcat_{\fC}} \to \groundcat^{\bijcat_{\fC}}$, coupled with the natural transformations $\mu : \freegraphmonad \freegraphmonad \Rightarrow \freegraphmonad$ and $\eta : \id \Rightarrow \freegraphmonad$, comprise a monad.
\end{proposition}
\begin{proof}
Graph substitution is associative and unital (\cite[Theorem 5.32; Lemma 5.31]{yj15}) which implies the result.
\end{proof}

\begin{definition}\label{def: csm fixed color}
Given an involutive set of colors $\mathfrak{C}$, the category of algebras over the monad $(\freegraphmonad_{\fC}, \mu, \eta)$ on $\groundcat^{\bijcat_{\fC}}$ is denoted by $\csm_{\fC}(\groundcat)$. Objects of $\csm_{\fC}(\groundcat)$ are called \emph{modular operads} in $\groundcat$ with objects $\fC$. 
\end{definition} 

Given a map $f: \fC \to \mathfrak{D}$ of involutive sets, there is corresponding adjoint pair
\[
	f_! : \csm_{\fC}(\groundcat) \leftrightarrows \csm_{\mathfrak{D}}(\groundcat) : f^*
\]
where $(f^*X)(S,\xi) = X(S,f\circ \xi)$.
It is evident that $(gf)^* = f^*g^*$, so it follows that $g_!f_! = (gf)_!$.

\begin{definition}
\label{definition CSM}
	Let $\csm(\groundcat)$ denote the category of all modular operads. If $X$ has objects $\fC$ and $Y$ has objects $\mathfrak{D}$, then 
	\[
		\csm(\groundcat)(X,Y) = \coprod_{f: \fC \to \mathfrak{D}} \csm_{\fC}(\groundcat)(X, f^*Y)
	\]
	where $f$ ranges over all maps of involutive sets $\fC \to \mathfrak{D}.$
	Composition of morphisms is as usual in the Grothendieck construction.
	More precisely, if $\iset$ is the category of involutive sets, then there is a functor $\iset^{\oprm} \to \cat$ that sends $\mathfrak{C}$ to $\csm_{\fC}(\groundcat)$ and $f : \fC \to \mathfrak{D}$ to $f^*$ defined above.
	Then $\csm(\groundcat) \to \iset$, sending a modular operad to its involutive set of colors, is the associated Grothendieck (cartesian) fibration.
\end{definition}

Each of the categories $\csm_{\fC}(\groundcat)$ will be complete or cocomplete when $\groundcat$ is.
Completeness is standard, while for cocompleteness one should check that $\freegraphmonad_{\fC}$ is a finitary monad.
In her thesis, Sophie Raynor shows that there is a colored operad whose category of algebras is $\csm_{\fC}(\groundcat)$ \cite[\S4.5]{Raynor:Thesis}, which implies this fact.

\begin{remark}
	Since each $f^*$ has a left adjoint $f_!$, the functor $\csm(\groundcat) \to \iset$ is actually a \emph{bifibration} (see, for instance, \cite[Lemma 9.1.2]{Jacobs:CLTT}).
	Given any bifibration with bicomplete base and bicomplete fibers, the total category is also bicomplete (this is classical, see Exercise 9.2.4, p.531 of \cite{Jacobs:CLTT}).
	Since $\iset$ and all $\csm_{\fC}(\groundcat)$ are bicomplete when $\groundcat$ is, it follows that $\csm(\groundcat)$ is also bicomplete when $\groundcat$ is.
\end{remark}

\begin{remark}\label{remark about joyal-kock}
The category of colored modular operads of Definition~\ref{definition CSM} was introduced in \cite{JOYAL2011105}, under the name `compact symmetric multicategories,' using a related monad but only for $\groundcat = \Set$.
One benefit to their approach is that it used a single monad, rather than one for each involutive set of colors. 
One drawback is that it is not clear how to generalize to the cases when $\groundcat$ is different from $\Set$.
Note that in the third paragraph of \S5 of \cite{JOYAL2011105}, the monad is not well-defined at level $n=0$; one needs to add in nodeless loops to the collection of graphs to make this correct.
An alternative approach can be found in \cite{Raynor:DLCSM}.
\end{remark}

\begin{remark}\label{remark about generators and relations definition}
At the beginning of the introduction, we said that (monochrome) modular operads may be specified by composition operations and by contraction operations, which satisfy a small collection of axioms.
Appropriate presentations appear in the non-skeletal setting in Definition 2.1 of \cite{Doubek:MECOA} (stable) and Definition A.4 of \cite{Markl:MEONMO} (unstable).
Of course this works just as well for the $\fC$-colored modular operads of Definition~\ref{def: csm fixed color}, with the understanding that one should replace finite sets with finite sets over $\fC$ and that compositions and contractions will be defined only when the colorings match; this was laid out in \cite[\S2.2]{Raynor:Thesis}.
All of these references cover the case of \emph{non-unital} modular operads.
To our knowledge there is not a similar presentation for the skeletal context (as in Remark~\ref{remark on skeletality}) in the literature.
However, for the case of cyclic operads (with units and involutive color sets), where we have compositions but no contractions, such a system is included in the paper \cite{DrummondColeHackney:DKHCO} of Drummond-Cole and the first author.
In any case, we expect that these types of `biased' definitions of modular operads would play a key role in determining whether modular operads are equivalent to strict inner Kan $\graphicalcat$-presheaves.
\end{remark}

\subsection{The modular operad associated to a graph}\label{sections: graphical csm}
Let us consider $\fC$-colored modular operads with underlying symmetric monoidal category $\groundcat = \Set$.
There is an adjunction
\[
	F_{\fC} \colon \Set^{\bijcat_{\fC}} \rightleftarrows \csm_{\fC} \colon U_{\fC},
\]
(where $U_{\fC}F_{\fC} = \freegraphmonad_{\fC}$)
which we can use to produce new modular operads.
In particular, if $G$ is a graph then we can produce a modular operad $\modopgenned{G}$ whose operations are generated by the vertices of $G$.

\begin{definition}[The modular operad generated by a graph]\label{definition_graphical_modular_operad}
Suppose $G$ is a connected,\footnote{This definition is nearly correct for disconnected graphs as well, but does not produce the expected answer when $G$ has more than one isolated vertex.} possibly unsafe, graph with set of arcs $A$ and set of vertices $V$.
\begin{enumerate}
\item If $\wp(A)$ is the power set of $A$, we consider the object $\underline{G}_{\wp}$ in $\Set^{\wp(A)}$ satisfying $\underline{G}_{\wp} (Z)$ is a point if $Z = i \nbhd(v)$ for some $v$, and otherwise $\underline{G}_{\wp} (Z)$ is empty.
\item The power set $\wp(A)$ of subsets of $A$ includes into the groupoid $\bijcat_A$ by sending a subset $Z \subseteq A$ to $(Z, \incl)$ (see Notation~\ref{notation inclusion}).
We write $\unnameddecoration{G} \in \Set^{\bijcat_A}$ for the left Kan extension of $\underline{G}_{\wp}$.
\item More concretely, $\unnameddecoration{G} \in \Set^{\bijcat_A}$ is given by
\[
	\unnameddecoration{G}(S,\xi)=
		\begin{cases} 
			\{(v,\xi)\} & \text{if $\xi : S \to A$ is injective and $\xi(S) = i\nbhd(v)$} \\ 
			\varnothing & \text{otherwise.} 
		\end{cases}
\]
\item Define a (free $A$-colored) modular operad, \emph{the modular operad generated by $G$}, as $\modopgenned{G}=F_A(\unnameddecoration{G}) \in \csm_A$.
\end{enumerate}
\end{definition} 

Given that an $\fC$-colored modular operad is an algebra over the monad $\freegraphmonad_{\fC}$ in Section~\ref{subsection CSM}, we see that an element in $\modopgenned{G}$ is represented by a $\unnameddecoration{G}$-decorated graph; let us unravel this a bit. 
If $(S,\xi)$ is an object of $\bijcat_A$, then 
\[ \modopgenned{G}(S,\xi) =  (\freegraphmonad \unnameddecoration{G})(S,\xi) = \colim_{(f,K,\zeta) \in \conngraph_{A}(S,\xi)} \unnameddecoration{G}[K,\zeta].\]
Here, $(K,\zeta)$ is an $A$-colored graph (which may be a nodeless loop, see Example~\ref{examples edge and loop}), $f: S \to \eth(K)$ is a bijection so that $\zeta|_{\eth(K)} \circ f = \xi$, and
\[ \unnameddecoration{G}[K,\zeta] =  \prod_{w\in V(K)} \unnameddecoration{G}(\dnbhd{w}, \zeta|_{\dnbhd{w}}). \]
Given the structure of $\unnameddecoration{G}$, the set $\unnameddecoration{G}[K,\zeta]$ will be a point just when, for each $w\in V(K)$, the function $\zeta|_{\dnbhd{w}}$ constitutes a bijection $\dnbhd{w} \to i\nbhd(v)$ for some (unique, since $G$ is connected) vertex $v\in V(G)$.
In all other cases, $\unnameddecoration{G}[K,\zeta]$ is the empty set.

\begin{remark}\label{remark embeddings}
Let $G$ be a safe graph.
An important special case of elements of $\modopgenned{G}$ come from embeddings in the sense of Definition~\ref{def nat trans}.
Specifically, if $K$ is an object of $\graphicalcat$ and $f : K\hookrightarrow G$ is an embedding, then the maps $A(K) \to A(G)$ and $V(K) \to V(G)$ constitute a $\unnameddecoration{G}$-decoration of $K$.
There's a slight ambiguity about where in $\modopgenned{G}$ to locate this element, and we make the following choice.
We have the factorization
\[ \begin{tikzcd}
\eth(K) \dar[hook] \rar{\cong} & \eth(f) \dar[hook] \\
A(K) \rar["f"] & A(G)
\end{tikzcd} \]
and we write $(f|_{\eth(K)})^{-1} : \eth(f) \to \eth(K)$ for inverse of the top map.
Then we associate to $f$ the object $((f|_{\eth(K)})^{-1}, K, f)$ in  $\conngraph_{A(G)}(\eth(f), \incl)$ (see Notation~\ref{notation inclusion} and the discussion following Definition~\ref{definition graph groupoids}), so we think of $f$ as representing an object of $\modopgenned{G}(\eth(f),\incl)$.
There are other choices about where this element should live; for example, we could have it live in $\modopgenned{G}(\eth(K),f|_{\eth(K)})$.
A primary benefit to our choice is that it is invariant under isomorphism of embeddings: if $f_1 : K_1 \to G$ and $f_2 : K_2 \to G$ are two embeddings with $f_1 = f_2z$ ($z$ an isomorphism), then $\eth(f_1) = \eth(f_2)$.
The isomorphism $z$ lives in $\conngraph_{A(G)}(\eth(f_1), \incl)$ so $f_1$ and $f_2$ will represent the same element of $\modopgenned{G}(\eth(f_1),\incl)$.
Had we made the alternative choice, where $f_j$ represents an element of $\modopgenned{G}(\eth(K_j), f_j|_{\eth(K_j)})$, then these two elements would not even be immediately comparable.

In summary, we've both shown how elements of $\bigembeddings(G)$ produce elements of $\modopgenned{G}$, and also that this association factors through $\embeddings(G)$.
That is, we have an inclusion 
\[
	\embeddings(G) \subseteq \coprod_{Z \subseteq A} \modopgenned{G}(Z,\incl).
\]
Be careful, though: if $Z \subseteq A$ is of order two, we may have distinct elements of $\modopgenned{G}(Z,\incl)$ which are both represented by embeddings, just as in \cite[\ONEembuniqueness]{modular_paper_one}.
\end{remark}

\begin{example}\label{example_trivial_graphical_modular}
Let $G = \exceptionaledge$ be the exceptional edge. 
We have $A=  \{ \edgemajor, \edgeminor \}$ and $V=\varnothing$.
Then $\unnameddecoration{G}$ is the initial object in $\Set^{\bijcat_A}$, that is, $\unnameddecoration{G}(S,\xi)=\varnothing$ for each finite set $S$ and each function $\xi : S \to A$.
As $F_A$ is a left adjoint, this implies that $\modopgenned{G} = F_A(\unnameddecoration{G})$ is initial in $\csm_A$.
The considerations above show that we have 
\[
	|\modopgenned{G}(S,\xi)| = \begin{cases}
		1 & \text{if $\xi : S \to A$ is bijective,} \\
		1 & \text{if $S=\varnothing$, and} \\
		0 & \text{otherwise}.
	\end{cases}
\]
The second line comes from the fact that there are two $A$-colorings of a nodeless loop, but they are isomorphic in $\conngraph_A(\varnothing, \incl)$.
We likewise have two $A$-colorings of the exceptional edge, which are isomorphic in $\conngraph_A$, but are incomparable once we consider the extra structure to make them objects of $\conngraph_A(S,\xi)$ for some $(S,\xi)$.
For any $P\in\csm_{\fC}$, we have $\csm(\modopgenned{\exceptionaledge}, P) \cong \fC$: any map $f$ is determined by $f(\edgemajor) \in \fC$.
\end{example}  

A nodeless loop will also generate the modular operad from this example, as the boundary of $G$ does not factor in the definition of $\modopgenned{G}$.

\begin{example} 
\label{example csm generated by medstarzero}
If $G$ is the isolated vertex $\medstar_0$, then we have $A = \varnothing$ and $V = \{v\}$. 
The resulting object $\modopgenned{\medstar_0}$ is in $\csm_{\varnothing}$, hence only has a single set to define. In this case, $\modopgenned{\medstar_0}(\varnothing, \id_{\varnothing})$ is a point.
In fact, $\csm_{\varnothing}$ is equivalent to the category of sets, and $\modopgenned{\medstar_0}$ is a generator.
\end{example} 

We wish to show that the assignment $G\mapsto \modopgenned{G}$ is the object part of a functor from $\graphicalcat\rightarrow \csm$.
As $\modopgenned{G}$ is a free $A$-colored modular operad, it is easy to define maps out of $\modopgenned{G}$.

\begin{lemma}\label{lemma data of map of graphical csm}
Suppose $G$ is a graph and $P$ is a $\fC$-colored modular operad.
A map
\[ f: \modopgenned{G} \to P \] is equivalent to the data: 
\begin{itemize} 
\item an involutive function $f_0: A\rightarrow \fC$, where $A$ is the set of arcs of $G$, and
\item for each vertex $v\in V(G)$, an element $f_1(v)$ in $P(i\nbhd(v), f_0|_{i\nbhd(v)})$.
\end{itemize} 
\end{lemma} 

\begin{proof} 
The first piece of data just comes from the fact that $A$ and $\fC$ are the color sets for these modular operads.
The data of a map $\modopgenned{G} \to P$ with underlying color map $f_0 : A \to \fC$ is just a map
$\modopgenned{G} \to f_0^*P$ in $\csm_A$.
But $\modopgenned{G}$ is free in $\csm_A$, so this just amounts to a map of $\bijcat_A$ diagrams $\unnameddecoration{G} \to f_0^*P$.
We of course have 
\[
	f_0^*P(S,\xi) = P(S,f_0\circ \xi).
\]
The result then follows from the description in Definition~\ref{definition_graphical_modular_operad} of $\unnameddecoration{G}$ as a left Kan extension.
\end{proof}

\begin{remark}[Composition of maps between graphical modular operads]\label{remark kleisli composition}
Let us describe composition of maps appearing in Lemma~\ref{lemma data of map of graphical csm} whose targets are also modular operads generated by graphs.
As might be expected, this looks a bit like Kleisli composition, but adjusted for the fact that $\csm$ is the Grothendieck construction associated to $\fC \mapsto \csm_{\fC}$ (Definition~\ref{definition CSM}).
	Specifically, suppose that $f : \modopgenned{G} \to \modopgenned{H}$ and $g: \modopgenned{H} \to \modopgenned{K}$ are modular operad maps.
	By the lemma, this is equivalent to maps
	\begin{align*}
		f_1: \unnameddecoration{G} &\to f_0^* \freegraphmonad_{A(H)} (\unnameddecoration{H}) \\
		g_1: \unnameddecoration{H} &\to g_0^* \freegraphmonad_{A(K)} (\unnameddecoration{K})
	\end{align*}
	in the diagram categories $\Set^{\bijcat_{A(G)}}$ and $\Set^{\bijcat_{A(H)}}$.
The map $g_0^*$ is a functor from $\Set^{\bijcat_{A(K)}}$ to $\Set^{\bijcat_{A(H)}}$ with $(g_0^*X)(S,\xi) = X(S,g_0\circ \xi)$; 
likewise, $g_0$ also induces a functor $g_0^*$ between $\csm_{A(K)} \to \csm_{A(H)}$ satisfying $U_{A(H)}g_0^* = g_0^*U_{A(K)}$.
Applying the first of these to the unit $\eta : \id \Rightarrow \freegraphmonad_{A(K)} = U_{A(K)}F_{A(K)}$ for the monad $\freegraphmonad_{A(K)}$ gives a natural transformation 
\[ g_0^* \Rightarrow g_0^*\freegraphmonad_{A(K)} = g_0^*U_{A(K)}F_{A(K)} = U_{A(H)}g_0^*F_{A(K)}. \]
Taking adjoints gives a natural transformation
	\begin{equation}\label{equation unlabeled arrow}
		F_{A(H)}g_0^* \Rightarrow g_0^*F_{A(K)}
	\end{equation}
of functors $\Set^{\bijcat_{A(K)}} \to \csm_{A(H)}$.
	To get $(g \circ f)_1$, we use the diagram in Figure~\ref{figure composition graphical modular}, 
where $\mu$ is the multiplication of the monad $\freegraphmonad_{A(K)}$.

\begin{figure}[t]
\[ \begin{tikzcd}[column sep=large]
\unnameddecoration{G} 
	\rar{f_1} \arrow[dddd, "(g\circ f)_1"]&  
f_0^* \freegraphmonad_{A(H)} (\unnameddecoration{H}) 
	\dar{f_0^* \freegraphmonad_{A(H)}(g_1)}  \\ 
	 & 
f_0^* \freegraphmonad_{A(H)}g_0^* \freegraphmonad_{A(K)} (\unnameddecoration{K}) 
	\dar{=}  \\ &
f_0^* U_{A(H)} F_{A(H)}g_0^* \freegraphmonad_{A(K)} (\unnameddecoration{K}) 
	\dar{\eqref{equation unlabeled arrow}} \\
& f_0^* U_{A(H)} g_0^*F_{A(K)} \freegraphmonad_{A(K)} (\unnameddecoration{K}) \dar{=} \\
(g_0\circ f_0)^* \freegraphmonad_{A(K)} (\unnameddecoration{K}) & 
f_0^* g_0^*\freegraphmonad_{A(K)} \freegraphmonad_{A(K)} (\unnameddecoration{K}) 
	\lar{(g_0 \circ f_0)^* \mu} 
\end{tikzcd}
\]
\caption{Composition of maps between graphical modular operads}\label{figure composition graphical modular}
\end{figure}
\end{remark}

We wish to extend the assignment $G\mapsto \modopgenned{G}$ to a functor $\graphicalcat\rightarrow \csm$. 
As we have seen, defining maps out of $\modopgenned{G}$ is straightforward, since $\modopgenned{G}$ is free in $\csm_{A}$. 
We use  Remark~\ref{remark embeddings} to regard embeddings as elements of $\modopgenned{G}$. 

\begin{definition}[Assignment on morphisms]\label{definition functor to CSM}
Suppose that $\varphi : G \to G'$ is a graphical map in $\graphicalcat$.
Define a morphism of modular operads $f : \modopgenned{G} \to \modopgenned{G'}$, using Lemma~\ref{lemma data of map of graphical csm}, as follows:
\begin{itemize}
\item The map of involutive sets $f_0 : A \to A'$ is just $\varphi_0$.
\item Each $\varphi_1(v) \in \embeddings(G')$ determines an element of $\modopgenned{G'}(\eth(\varphi_1(v)),\incl)$ by Remark~\ref{remark embeddings}.
There is an isomorphism $(\eth(\varphi_1(v)),\incl) \cong (i\nbhd(v),\varphi_0|_{i\nbhd(v)})$ in $\bijcat_{A'}$ by Definition~\ref{graphical category definition}\eqref{graphical map defn: boundary compatibility}, and we let
\[
	f_1(v) \in \modopgenned{G'}(i\nbhd(v),\varphi_0|_{i\nbhd(v)}) \cong \modopgenned{G'}(\eth(\varphi_1(v)),\incl)
\]
be the element corresponding to $\varphi_1(v) \in \embeddings(G')$.
\end{itemize}
\end{definition}

Each isomorphism of $\graphicalcat$ maps to an isomorphism of modular operads.
In Lemma~\ref{lemma on iso classes} we give a partial converse to this statement.
Notice in this lemma that the graphs $G$ and $G'$ are in $\graphicalcat$; in particular, neither of these graphs is a nodeless loop.
Of course nodeless loops will generate the same modular operad as the exceptional edge (the initial object in $\csm_{2\{*\}}$ as in Example~\ref{example_trivial_graphical_modular}), though these are not isomorphic graphs.
See the paragraph preceding Remark~\ref{remark egc nerve theorem} where we consider this extension.
A discussion on a similar topic in the directed setting appears in Section 2 of \cite{hry_factorizations}.

\begin{lemma}\label{lemma on iso classes}
Suppose that $G$ and $G'$ are graphs in $\graphicalcat$.
	If $f: \modopgenned{G} \to \modopgenned{G'}$ is an isomorphism of modular operads, then there exists an isomorphism $\varphi : G \to G'$ in $\graphicalcat$ so that $\varphi \mapsto f$.
\end{lemma}
\begin{proof}
As we know $f_0 : A \to A'$ is an involutive isomorphism, we replace (strictly for convenience) $G'$ with an isomorphic graph $H$ which has the same set of arcs as $G$ and the same vertices as $G'$.
\[ \begin{tikzcd}
\modopgenned{G} \rar["g"] \arrow[dr, "f" swap] & \modopgenned{H} \arrow[d, "\cong" swap] \\
& \modopgenned{G'}
\end{tikzcd} \]
It is sufficient to show that the induced isomorphism $g : \modopgenned{G} \to \modopgenned{H}$ comes from an isomorphism in $\graphicalcat$.
Let $h$ be the inverse to $g$.

We are now just working in $\csm_{A}$, the category of algebras for $\freegraphmonad_A = \freegraphmonad$.
The composition diagram in Remark~\ref{remark kleisli composition} simplifies to the usual Kleisli composition diagrams.
\[ 
\begin{tikzcd}
\unnameddecoration{G} \rar{g_1} \dar[swap]{(h\circ g)_1} & \freegraphmonad(\unnameddecoration{H}) \dar{\freegraphmonad(h_1)} \\
\freegraphmonad (\unnameddecoration{G}) & \freegraphmonad \freegraphmonad (\unnameddecoration{G}) \lar["\mu_{\unnameddecoration{G}}"] 
\end{tikzcd} 
\qquad \& \qquad
\begin{tikzcd}
\unnameddecoration{H} \rar{h_1} \dar[swap]{(g\circ h)_1} & \freegraphmonad(\unnameddecoration{G}) \dar{\freegraphmonad(g_1)} \\
\freegraphmonad (\unnameddecoration{H}) & \freegraphmonad \freegraphmonad (\unnameddecoration{H}) \lar["\mu_{\unnameddecoration{H}}"] 
\end{tikzcd} 
\]
By the assumption that $h= g^{-1}$, we have that $(h\circ g)_1 = \eta_{\unnameddecoration{G}}$ and $(g\circ h)_1 = \eta_{\unnameddecoration{H}}$, where $\eta$ is the unit of the monad.

Suppose that $v$ is a vertex of $G$; then the map
\[
	g_1 : \unnameddecoration{G}(i\nbhd(v), \incl) \to \freegraphmonad(\unnameddecoration{H})(i\nbhd(v), \incl)
\]
takes $(v,\incl)$ (see Definition~\ref{definition_graphical_modular_operad}) to some $\unnameddecoration{H}$-decorated graph $K$.
The (connected) graph $K$ must have at least one vertex, since $(h\circ g)_1(v,\incl)$ is a star, thus is not an edge.
Similarly, in the right-hand diagram we have that $h_1$ only produces graphs that have at least one vertex.

Now if $g_1(v,\incl)$ has more than one vertex, or a loop at a vertex, then so does $(h\circ g)_1(v,\incl)$ since $h_1$ does not send vertices to edges and thus these are preserved by $\mu$. 
Since $(h\circ g)_1(v,\incl)$ is $(\medstar_{i\nbhd(v)}, \xi^\medstar)$, we thus know that $g_1(v,\incl)$ is a $\unnameddecoration{H}$-decorated star.
Likewise, $h_1$ produces only $\unnameddecoration{G}$-decorated stars.

It follows that $g_1$ and $h_1$ induce a bijection between $V(G)$ and $V(H)$.
If $w$ is the vertex of $H$ that is associated to $g_1(v)$, then we have $i : \nbhd(w) \to i\nbhd(v)$ is a bijection, so $\nbhd(v) = \nbhd(w)$.
Thus $g$ induces an isomorphism $G\cong H$ of graphs.
\end{proof}

\begin{proposition}\label{proposition functor graphical to csm} 
The assignment $G \mapsto \modopgenned{G}$ on objects from Definition~\ref{definition_graphical_modular_operad} and the assignment on graphical maps from Definition~\ref{definition functor to CSM} constitute a faithful functor $J : \graphicalcat \to \csm$ which is injective on isomorphism classes of objects.
\end{proposition} 

\begin{proof}
The fact that $\graphicalcat\to \csm$ is a functor follows by comparing Figure~\ref{figure composition graphical modular} from Remark~\ref{remark kleisli composition} with the composition for $\graphicalcat$ (Definition~\ref{def graph map composition}).
Lemma~\ref{lemma on iso classes} shows that the functor is injective on isomorphism classes of objects.
To see that the functor is faithful, suppose that $\varphi$ and $\psi$ are elements of $\graphicalcat(G,G')$ which map to the same morphism $f : \modopgenned{G} \to \modopgenned{G'}$.
Then the maps on color sets $\varphi_0$ and $\psi_0$ are equal.
Further, for each $v\in V(G)$ the element $f_1(v)$ in \[ \embeddings(G') \subseteq \coprod_{Z \subseteq A'} \modopgenned{G'}(Z,\incl),\] is equal to both $\varphi_1(v)$ and $\psi_1(v)$.
\end{proof} 

\begin{example}\label{example about graphicalcat not being full}
The functor $J: \graphicalcat \to \csm$ is not full.
Here we give two examples.
\begin{itemize}
	\item 
Consider the two graphs from Figure~\ref{figure: etale not embedding}.
There is a map from $\modopgenned{G}$ to $\modopgenned{G'}$ sending generators to generators, where each $v_j$ goes to $v$ and each $w_j$ goes to $w$, but there is no graphical map $G \to G'$ which has this behavior.
This example was explained to us by J. Kock, as an illustration of the difference between \'etale and embedding.
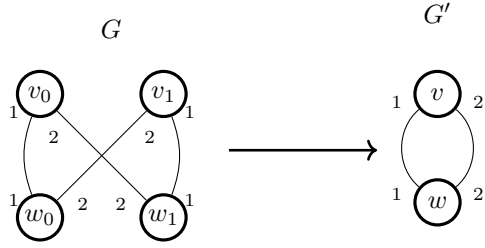
\begin{figure}
\begin{center}
\begin{tikzpicture}


\node [plain] (v) {$v_0$};
\node [plain, below=1cm of v] (v2) {$w_0$};
\node [plain, right=1cm of v] (v3) {$v_1$};
\node [plain, below=1cm of v3] (v4) {$w_1$};

\node[right=.5 cm of v] (q) {};
\node [above=.5cm of q] () {$G$};
\draw [-, bend right=20] (v) to node[black, very near end, swap]{\scriptsize{$1$}} node[black, very near start, swap]{\scriptsize{$1$}}   (v2); 
\draw [-, bend left=20] (v3) to node[black, very near end]{\scriptsize{$1$}} node[black, very near start]{\scriptsize{$1$}}  (v4);
\draw [-] (v3) to node[black, very near end]{\scriptsize{$2$}} node[black, very near start]{\scriptsize{$2$}}  (v2);
\draw [-] (v) to node[black, very near end, swap]{\scriptsize{$2$}} node[black, very near start, swap]{\scriptsize{$2$}}  (v4);

\node[right=.3 cm of v3] (q2) {};
\node[below=.5 cm of q2] (q3) {};
\node[right=2 cm of q3] (q4) {};

\node [plain, right=3cm of v3] (s2) {$v$};
\node [plain, below=.8cm of s2] (w) {$w$};
\node [above=.5cm of s2] () {$G'$};

\draw [-, bend right =50] (s2) to  node[black, very near end, swap]{\scriptsize{$1$}} node[black, very near start,swap]{\scriptsize{$1$}}  (w);
\draw [-, bend left=50] (s2) to node[black, very near end]{\scriptsize{$2$}} node[black, very near start]{\scriptsize{$2$}}  (w);

\draw [black, ->, line width= 1pt, shorten >=0.03cm] (q3) to  (q4);

\end{tikzpicture}
\end{center}
\caption{An \'etale map that is not an embedding.}\label{figure: etale not embedding}
\end{figure}
\item There is a map $\modopgenned{\medstar_0} \to \modopgenned{\exceptionaledge}$ which takes the unique vertex of $\medstar_0$ (see Example~\ref{example csm generated by medstarzero}) to the unique element in $\modopgenned{\exceptionaledge}(\varnothing, \incl)$ (see Example~\ref{example_trivial_graphical_modular}).
In contrast, there are no maps $\medstar_0 \to \exceptionaledge$ in $\graphicalcat$.
\end{itemize}
\end{example}

\section{The nerve theorem}\label{section: nerve theorem}
At this point, we have defined a functor $J: \graphicalcat \to \csm$.
One can consider the associated singular functor, or \emph{nerve functor}, which is specified by $N(-) = \hom(J,-)$ and goes from $\csm$ to the category of $\graphicalcat$-presheaves.
The aim of this section is to prove Theorem~\ref{nerve_theorem}, which says both that $N$ is fully-faithful and identifies the essential image.

\begin{definition}\label{definition_modular_nerve}
The \emph{nerve functor for modular operads} is the functor \[\begin{tikzcd} \csm \arrow[r,"N"] & \pregraphcat\end{tikzcd}\]
which is given on a modular operad $P$ and a graph $G\in\graphicalcat$ by
\[NP_{G}= \csm(\modopgenned{G},P). \] Here $\modopgenned{G}$ is the modular operad generated by $G \in \graphicalcat$ (Definition~\ref{definition_graphical_modular_operad}).
\end{definition}

An element of the set $NP_{G}$ is a $P$-\emph{decoration} of the graph $G$ (Definition~\ref{defn obj of decorations}).
This description comes directly from the fact that $NP_{G}=\csm(\modopgenned{G},P)$ and the description of a graphical map given in Lemma~\ref{lemma data of map of graphical csm}.
\begin{remark}
Given a graph $G\in \graphicalcat$, we now have two ways to assign an object of $\pregraphcat$ to $G$.
The first is to consider the representable presheaf $\rgc[G]$, while the second is to first consider the modular operad $\modopgenned{G}$ and then take the nerve.
In light of Example~\ref{example about graphicalcat not being full}, we do not expect them to always be the same.
The representable $\rgc[G]$ is always a subobject of $N\modopgenned{G}$ (since $J$ is faithful), but, in fact, they nearly never coincide.
To see this, let $K$ be the loop with one vertex and let $k$ be one of the two arcs of $K$.
By Lemma~\ref{lemma data of map of graphical csm}, for each arc $a$ of $G$ there is map
\(
	f : \modopgenned{K} \to \modopgenned{G}
\)
which sends $k$ to $a$ and the unique vertex of $K$ to the edge spanned by $a$.
This type of collapse behavior is precisely what is prohibited by \eqref{graphical map defn: collapse condition} of Definition~\ref{graphical category definition}.
Thus the inclusion $\rgc[G]_K \subseteq N\modopgenned{G}_K$ is strict as long as the arc set of $G$ is non-empty.
On the other hand, we have $\rgc[\medstar_0] = N\modopgenned{\medstar_0}$.
\end{remark}

\begin{remark}\label{remark nary operations pullback}
Suppose we are given an object $(S,\xi)$ of $\bijcat_{\fC}$, and let $\medstar_{S}$ be the graph from Definition~\ref{definition star}.
Recall from Example~\ref{examples edge and loop} that we write $A(\exceptionaledge) = \eearcs$ and define, for each $s\in S$, an embedding 
$h_s : \exceptionaledge \to \medstar_S$ which sends $\edgemajor$ to $s$.
There is a natural map 
\[
	\ell_S : NP_{\medstar_S} \to (NP_{\exceptionaledge})^S
\]
which takes an element $x$ to the function $(s\mapsto x\circ h_s) \in \hom(S,NP_{\exceptionaledge})$.
Under the identifications $NP_\exceptionaledge = \csm(\modopgenned{\exceptionaledge}, P) = \hom(\{\edgemajor\}, \fC) = \fC$, we may regard the function $\xi : S \to \fC$ as an element of the codomain of $\ell_S$.
The preimage of $\xi$ under $\ell_S$ is precisely $P(S,\xi)$.
That is, $P(S,\xi)$ is part of the following pullback diagram.
\[
\begin{tikzcd}
	P(S,\xi) \dar \rar & NP_{\medstar_S}\arrow[d, "\ell_S"] \\
	\ast \rar{\xi} &  (NP_{\exceptionaledge})^S \cong \fC^{|S|}
\end{tikzcd}
\]
We will use this frequently in what follows.
\end{remark}

The Segal core inclusions $\segalcore[G] \hookrightarrow \rgc[G]$, from Definition~\ref{def segal core and graph sub}, are induced by the embeddings $\iota_v : \medstar_v\rightarrow G$ (Definition~\ref{definition star}).
This allows us to give the following generalization of the classical Segal condition for categories.

\begin{definition}[Segal objects]\label{definition_Segal_condition}
Suppose that $X$ is a $\graphicalcat$-presheaf (in $\Set$).
\begin{itemize}
	\item The \emph{Segal map at $G$} is the map 
\[\begin{tikzcd}
\pregraphcat(\rgc[G], X) \rar & \pregraphcat(\segalcore[G], X)\end{tikzcd}\] 
induced by the Segal core inclusion $\segalcore[G] \hookrightarrow \rgc[G]$.
\item The presheaf $X$ is said to satisfy the \emph{Segal condition} if, for every $G\in\graphicalcat$, the Segal map at $G$ is a bijection.
\end{itemize}
\end{definition} 

The reader familiar with the work of Chu and Haugseng may wonder about the relation of this definition with \cite[Definition 2.6]{ChuHaugseng:HCAvSC}.
As observed in Example 3.12 of \cite{ChuHaugseng:HCAvSC}, the category $\graphicalcat^{\oprm}$ admits the structure of an `algebraic pattern.' 
A presheaf $X$ satisfies the Segal condition (in our sense) if and only if it is a `$\graphicalcat^{\oprm}$-Segal object in $\Set$.'  

\begin{remark}
If $X$ is an object of $\pregraphcat$ and $G$ is $\medstar_n$ or $\exceptionaledge$, then the Segal map at $G$ is a bijection.
\end{remark}

We are now prepared to state the first main theorem of this paper.
\begin{theorem}
\label{nerve_theorem}
The nerve functor $N: \csm \to \pregraphcat$ is fully faithful.
Moreover, the following statements are equivalent for $X\in \pregraphcat$.
\begin{enumerate} 
\item There exists a modular operad $P$ and an isomorphism $X\cong NP$.
\item $X$ satisfies the Segal condition.
\end{enumerate} 
\end{theorem} 

We will need a bit of scaffolding before we can approach the proof of this theorem, which appears below.

Suppose that $G$ is a graph with at least one vertex.
As in Definition~\ref{def segal core and graph sub}, for each internal edge $e\in E_i$, we choose an ordering $e = [x_{e}^1, x_{e}^2]$ for the two-element equivalence class of arcs comprising $e$.
\begin{itemize}
	\item Write $\outeredge_e : \exceptionaledge \to \medstar_{tx_e^1}$ for the embedding that sends $\edgemajor$ to $(x_e^1)^\dagger \in \eth(\medstar_{tx_e^1})$.
	\item Write $\inneredge_e : \exceptionaledge \to \medstar_{tx_e^2}$ for the embedding that sends $\edgemajor$ to $x_e^2 \in D(\medstar_{tx_e^2})$.
\end{itemize}

\begin{lemma}
\label{nerve layers as equalizer}
Suppose that $P$ is a modular operad and $G$ is a graph with at least one vertex. 
There is an equalizer diagram
\begin{equation} \label{nerve equalizer eq} 
\begin{tikzcd}
NP_G \rar & \prod\limits_{v\in V} NP_{\medstar_v} \rar[shift left, "\outeredge^*"] \rar[shift right, "\inneredge^*" swap] & \prod\limits_{e\in E_i} NP_{\exceptionaledge}
\end{tikzcd} \end{equation}
where $\outeredge^*$ and $\inneredge^*$ are defined so that the diagrams
\[ 
\begin{tikzcd}
\prod\limits_{v\in V} NP_{\medstar_v} \rar["\outeredge^*"] \dar["\pi_{tx_e^1}"] & \prod\limits_{e\in E_i} NP_{\exceptionaledge} \dar["\pi_e"] \\
NP_{\medstar_{tx_e^1}} \rar["NP_{\outeredge_e}"]& NP_{\exceptionaledge}
\end{tikzcd} 
\qquad
\begin{tikzcd}
\prod\limits_{v\in V} NP_{\medstar_v} \rar["\inneredge^*"] \dar["\pi_{tx_e^2}"] & \prod\limits_{e\in E_i} NP_{\exceptionaledge} \dar["\pi_e"] \\
NP_{\medstar_{tx_e^2}} \rar["NP_{\inneredge_e}"]& NP_{\exceptionaledge}
\end{tikzcd} 
\]
commute for each $e\in E_i$.
\end{lemma}
\begin{proof}
Combine Remark~\ref{remark nary operations pullback} with  Lemma~\ref{lemma data of map of graphical csm}.
\end{proof}

\begin{lemma}\label{Lemma_nerve_statisfies_Segal}
The nerve of a modular operad satisfies the Segal condition.
\end{lemma} 
\begin{proof} 
If $P$ is a modular operad, then 
\begin{equation}\label{coeq eq exchange}
\pregraphcat(\segalcore[G], NP)= \pregraphcat(\coequalizer(\outeredge,\inneredge),NP) \cong \equalizer\Big( \pregraphcat(\outeredge, NP), \pregraphcat(\inneredge, NP) \Big).\end{equation} 
Here $\pregraphcat(\outeredge, NP)$ is the top map of the following commutative diagram
\[ \begin{tikzcd}
\pregraphcat\Big(\coprod\limits_{v\in V} \rgc[\medstar_v], NP\Big) \rar{\outeredge} \dar{\cong} & 
\pregraphcat\Big(\coprod\limits_{e \in E_i} \rgc[\exceptionaledge], NP\Big) \dar{\cong} \\
\prod\limits_{v\in V} NP_{\medstar_v} \rar["\outeredge^*"] & \prod\limits_{e\in E_i} NP_{\exceptionaledge}
\end{tikzcd} \]
and likewise for $\pregraphcat(\inneredge, NP)$.
By Lemma~\ref{nerve layers as equalizer}, the equalizer in \eqref{coeq eq exchange} coincides with $NP_G$.
\end{proof} 

Let us now verify the first statement in Theorem~\ref{nerve_theorem}.

\begin{proposition}\label{nerve_fully_faithful}
The nerve functor $N : \csm \to \pregraphcat$ 
is fully faithful.
\end{proposition}

\begin{proof} 
Throughout, let $P$ be in $\csm_{\fC}$ and $Q$ be in $\csm_{\mathfrak{D}}$.
First we will show that the nerve functor is faithful.
Suppose we are given $f,g: P\rightarrow Q$ in $\csm$ with the property that $Nf=Ng$.
In particular, $f$ and $g$ are equal as involutive functions from $NP_{\exceptionaledge}=\fC$ to $ NQ_{\exceptionaledge}=\mathfrak{D}$.
As we mentioned in Remark~\ref{remark nary operations pullback}, the set $P(S,\xi)$ is the pullback of
\[\begin{tikzcd} NP_{\medstar_S}\arrow[r] &  (NP_{\exceptionaledge})^S&  \ast \arrow[l,  "\xi" swap] \end{tikzcd}\] and similarly for $Q(S,f\circ \xi) = Q(S,g\circ \xi)$.
As we have a commutative diagram
\[\begin{tikzcd} NP_{\medstar_S}\arrow[r]\arrow[d, "Nf_{\medstar_S}=Ng_{\medstar_S}" description] &  
(NP_{\exceptionaledge})^S \arrow[d, "f_0^S= g_0^S"]&  \ast \arrow[l,"
\xi" swap]\arrow[d,"="] \\
NQ_{\medstar_S}\arrow[r] &  
(NQ_{\exceptionaledge})^S &  \ast \arrow[l,"
f_0\xi" swap],
 \end{tikzcd}\] it follows that $f=g$ on each set $P(S,\xi)$.
Thus $f$ and $g$ are identical.

To show that the nerve functor is full, now suppose we have a map $\tilde{f}: NP\rightarrow NQ$ in $\pregraphcat$.
We wish to exhibit a modular operad map $f:P\rightarrow Q$ so that $\tilde{f} = Nf$.
By definition, the map $\tilde{f}:NP_{\exceptionaledge}\rightarrow NQ_{\exceptionaledge}$ is a map of involutive sets $f_0:\fC\rightarrow\mathfrak{D}$.

Similar to previous argument, we know that for each $m$ we have a map of diagrams 
\[\begin{tikzcd} NP_{\medstar_S}\arrow[r]\arrow[d, "\tilde{f}"] &  
(NP_{\exceptionaledge})^S \arrow[d, "f_0^S"]&  \ast \arrow[l,"\xi" swap]\arrow[d,"="] \\
NQ_{\medstar_S}\arrow[r] &  
(NQ_{\exceptionaledge})^S &  \ast \arrow[l, "f_0\xi" swap],
 \end{tikzcd}\]
which induces a map of pullbacks $f:P(S,\xi)\rightarrow Q(S,f_0\xi)$.

We've now defined a map $f$ in $\groundcat^{\bijcat_{\fC}}$ from $P$ to $f_0^* Q$.
It remains to show that $f$ is modular operad map, at which point it is automatic that $Nf=\tilde{f}$.
This amounts to showing that the diagram 
\[\begin{tikzcd}
\freegraphmonad_{\fC}P(S,\xi)\arrow[r, "\freegraphmonad_{\fC}f"]\arrow[d] & \left(f_0^*\freegraphmonad_{\mathfrak{D}}Q\right)(S,\xi) \rar{=} \dar & \freegraphmonad_{\mathfrak{D}}Q(S,f_0\xi)\dar \\
P(S,\xi)\arrow[r,"f"]& (f_0^*Q)(S,\xi) \rar{=} &Q(S,f_0\xi).
\end{tikzcd} \] commutes, where the vertical maps $\freegraphmonad_{\fC}P \to P$ and $\freegraphmonad_{\mathfrak{D}}Q \to Q$ are the algebra structure maps for $P$ and $Q$.

Consider an object $(h,G,\zeta) \in \conngraph_{\fC}(S,\xi)$; that is, $(G,\zeta)$ is a $\fC$-colored graph and $h : S \to \eth(G)$ is a bijection so that $\zeta|_{\eth(G)} \circ h = \xi$.
It suffices, by Definition~\ref{defn obj of decorations}, to show that for any such object $(h,G,\zeta)$ that the diagram
\[ 
\begin{tikzcd}
P[G,\zeta]\arrow[r]\arrow[d] & Q[G, f_0\zeta]\arrow[d]\\
P(S,\xi)\arrow[r,"f"]& Q(S,f_0\xi).
\end{tikzcd} \] commutes.
This is a automatic, as this is a sub-diagram of
\[\begin{tikzcd}
NP_{G}\arrow[r]\arrow[d] & NQ_{G}\arrow[d]\\
NP_{\medstar_S}\arrow[r,"f"]& NQ_{\medstar_S}
\end{tikzcd}\]
where $\medstar_S\rightarrow G$ is the active map determined by $h : S \to \eth(G)$ (Definition~\ref{def active and star active}).
\end{proof} 

\subsection{The modular operad associated to a Segal presheaf}
\label{section mod op segal presheaf}
As we saw in Lemma~\ref{Lemma_nerve_statisfies_Segal}, the nerve functor factors through the full subcategory of Segal presheaves.
We now turn to the last remaining part of Theorem~\ref{nerve_theorem}, namely that every Segal presheaf is (up to isomorphism) the nerve of a modular operad.
This requires a construction $X \rightsquigarrow \csmx_X$ taking a Segal presheaf to a modular operad.

It is technically convenient to work with the extended graphical category $\egc$ in this section.
In a moment, we will fix a Segal $\egc$-presheaf $X$ and endeavor to define $\csmx = \csmx_X$, which we call the \emph{modular operad associated to $X$} (Definition~\ref{def modular operad associated to X}).
As the underlying object of $\csmx$ is defined via certain pullbacks (Definition~\ref{def csmx underlying object}), our construction will only produce an isomorphism class, and is invariant under isomorphism of $\egc$-presheaves.
Thus the following remark is harmless.

\begin{remark}
\label{remark graphcat vs egc segal and associated}
If $Y$ is a Segal $\graphicalcat$-presheaf, then its right Kan extension $\iota_*Y$ along the inclusion $\iota : \graphicalcat \to \egc$ is a Segal $\egc$-presheaf \cite[\ONEsegaltosegal]{modular_paper_one}.
By definition, the \emph{modular operad associated to $Y$} is just the modular operad associated to the Segal $\egc$-presheaf $\iota_*Y$ (Definition~\ref{def modular operad associated to X}).
On the other hand, if $X$ is a Segal $\egc$-presheaf, then its restriction $\iota^*X$ is a Segal $\graphicalcat$-presheaf and $X\cong \iota_*\iota^*X$.
Thus the modular operad associated to $X$ is the the same as the modular operad associated to the restriction $\iota^*X$.
\end{remark}

\begin{notation}
\label{notation segal core subscript}
If $G$ is a graph in $\egc$, then we will write $\segalcore[G] \subseteq \regc[G]$ for the relevant subobject of the representable object. 
When $G$ a safe graph, this is just the left Kan extension of the usual inclusion $\segalcore[G] \to \rgc[G]$ (Definition~\ref{def segal core and graph sub}), while if $G$ is a nodeless loop then it is of the form $\regc[\exceptionaledge] \to \regc[G]$.
If $X$ is a $\egc$-presheaf, we write
\[
	X_{\segalcore[G]} = \hom(\segalcore[G], X).
\]
\end{notation}

Fix an arbitrary $\egc$-presheaf $X$ satisfying the Segal condition, and let $\fC$ be the involutive set $X_{\exceptionaledge}$.
We start by constructing the underlying object $\csmx$ in $\Set^{\bijcat_{\fC}}$.

\begin{definition}
\label{def csmx underlying object}
For each function $\xi$ from a set $S$ to $X_{\exceptionaledge} = \fC$ we define a set $\csmx(S,\xi)$ as the pullback of 
\[\begin{tikzcd} X_{\medstar_S}\arrow[r] &  (X_{\exceptionaledge})^S&  \ast \arrow[l, "\xi" swap].
\end{tikzcd}\] 
This defines an object $\csmx$ in $\Set^{\bijcat_{\fC}}$. 
\end{definition}

In particular, $\csmx(\varnothing, \incl)$ is isomorphic to $X_{\medstar_0}$.

In order to now exhibit the object $\csmx$ as an $\freegraphmonad_{\fC}$-algebra, we need to produce a map $\gamma: \freegraphmonad\csmx\to \csmx$.
We again follow the notation that was introduced just before Definition~\ref{defn obj of decorations} and abbreviate, for a $\fC$-colored graph $(G,\zeta)$
\begin{equation}\label{eq X zetavnbhd}
	X(\zetavnbhd) = X(\dnbhd{v}, \zeta|_{\dnbhd{v}}) \qquad \& \qquad \conngraph_{\fC}(\zetavnbhd) = \conngraph_{\fC}(\dnbhd{v}, \zeta|_{\dnbhd{v}}).
\end{equation}
Let $(f,G,\zeta)$ be an object of $\conngraph_{\fC}(S,\xi)$ (that is, $(G,\zeta)$ is a $\fC$-colored graph and $f: S \to \eth(G)$ is a bijection with $\zeta|_{\eth(G)} \circ f = \xi$).
Since $X$ satisfies the Segal condition and $\csmx(\zetavnbhd)$ is a subset of $X_{\medstar_v}$, we have an inclusion
\[ \csmx[G,\zeta] = \prod_{v\in V} \csmx(\zetavnbhd) \hookrightarrow X_G. \]
Note that when $G$ has an empty vertex set, then $\csmx[G,\zeta]$ is a one-point set and this inclusion is essentially equivalent to the coloring $\zeta$.

\begin{definition}[Action on $\csmx$]
\label{def of gamma restriction}
We define the algebra structure map $\gamma: \freegraphmonad\csmx \to \csmx$.
\begin{itemize}
\item Suppose that $(f,G,\zeta)$ is an object of $\conngraph_{\fC}(S,\xi)$. 
We have the following commutative diagram 
\[\begin{tikzcd}
\csmx[G, \zeta] \arrow[drr, bend left] \arrow[dr, dotted, "\exists!"]\arrow[dd, bend right]&& \\ 
& \csmx(S,\xi) \arrow[r]\arrow[d]	& \ast \arrow[d, "\xi"] \\
X_G \rar["d"] & X_{\medstar_S} \arrow[r] & (X_{\exceptionaledge})^S
\end{tikzcd}\] 
where the bottom square is the pullback used to define $\csmx(S,\xi)$ and the map $d:X_G \rightarrow X_{\medstar_S}$ is induced by the active map $\medstar_S \to G$ coming from the bijection $f: S \to \eth(G)$.
We write $\gamma_{f,G,\zeta}$ for the induced map $\csmx[G, \zeta] \to \csmx(S,\xi)$.
\item
Since $\freegraphmonad\csmx (S,\xi)=\colim_{\conngraph_{\fC}(S,\xi)}\csmx[G,\zeta]$ we have defined a map $\freegraphmonad\csmx(S,\xi)\rightarrow \csmx(S,\xi)$ on each component of the colimit and this can be extended to the whole colimit.
Since $(S,\xi)$ was arbitrary, we have a map 
\[
\gamma :\freegraphmonad\csmx\rightarrow \csmx
\]
in $\bijcat_{\fC}$.
\end{itemize}
\end{definition}

In other words, the structure map $\gamma$ is ultimately induced by the composites (see Notation~\ref{notation segal core subscript})
\[ \begin{tikzcd}
\csmx[G,\zeta] \subseteq X_{\segalcore[G]} & X_G \rar \lar["\cong" swap] & X_{\medstar_G}
\end{tikzcd} \]
where $\medstar_G = \medstar_{\eth(G)} \to G$ is the active map induced by the identity on $\eth(G)$.

It remains to show that the map $\gamma$ from Definition~\ref{def of gamma restriction} turns $\csmx$ into a $\freegraphmonad$-algebra.
Let us first address the unit axiom.
\begin{lemma}
\label{lemma gamma unit}
The diagram 
\[ \begin{tikzcd}
\csmx \rar{\eta_{\csmx}} \arrow[dr, "=" swap] & \freegraphmonad \csmx \dar{\gamma} \\
& \csmx 
\end{tikzcd} \]
commutes.
\end{lemma}
\begin{proof}
The composite is 
\begin{equation}
\label{eq some composite}
\begin{tikzcd}[column sep=huge]
\csmx(S,\xi) \rar["=", "\eta_{(S,\xi)}" swap] & \csmx[\medstar_S, \xi^\medstar] \rar["\gamma_{\id_S, \medstar_S, \xi^\medstar}" swap] & \csmx(S,\xi)
\end{tikzcd}
\end{equation}
Of course when $G=\medstar_S$ and $d$ is the identity in Definition~\ref{def of gamma restriction}, then the identity map makes the diagram commute (hence is the unique map making the diagram commute).
It follows that the composite \eqref{eq some composite} is the identity on $\csmx(S,\xi)$.
\end{proof}

It remains to show that the diagram
\begin{equation}\label{heart}
\begin{tikzcd}
\freegraphmonad\freegraphmonad\csmx \arrow[r, "\freegraphmonad\gamma"] \arrow[d, "\mu"] & \freegraphmonad\csmx \arrow[d, "\gamma"]\\
\freegraphmonad\csmx \arrow[r, "\gamma"] &  \csmx 
\end{tikzcd}
\end{equation} 
commutes, where $\gamma : \freegraphmonad\csmx \to \csmx$ is the proposed $\freegraphmonad$-algebra structure map from Definition~\ref{def of gamma restriction}.
It is enough to show, for each $\fC$-colored graph $(G,\zeta)$, that the diagram commutes when restricted to $(\freegraphmonad\csmx)[G,\zeta] \to (\freegraphmonad\freegraphmonad\csmx)(\eth(G), \incl)$ coming from considering $(\id_{\eth(G)}, G,\zeta)$ as an object of $\conngraph_{\fC}(\eth(G),\incl)$.
\begin{lemma}
\label{lemma no vertex algebra lemma}
If $G$ has no vertices (that is, if $G$ is a nodeless loop or the exceptional edge), then the diagram 
\[ \begin{tikzcd}[column sep=small, row sep=small]
(\freegraphmonad\csmx)[G,\zeta] \arrow[rr] \arrow[dd] \arrow[rrrddd] & & \csmx[G,\zeta]\arrow[dr]  \\
& & & (\freegraphmonad\csmx)(\eth(G), \incl) \\
\csmx[G,\zeta] \arrow[dr] \\
& (\freegraphmonad\csmx)(\eth(G), \incl) & & (\freegraphmonad\freegraphmonad\csmx)(\eth(G), \incl) \arrow[uu, "\freegraphmonad \gamma"] \arrow[ll, "\mu"]
\end{tikzcd} \]
commutes.
Here, the diagonal maps are the structural maps for the colimits and the top and left maps are just the unique map on the point.
This implies that \eqref{heart} commutes when restricted to $(\freegraphmonad\csmx)[G,\zeta]$.
\end{lemma}
\begin{proof}
The long diagonal followed by either $\freegraphmonad \gamma$ or $\mu$ factors through the structural map \[
\csmx [G,\zeta] \to \colim_{(f,K,\zeta') \in \conngraph_{\fC}(\eth(G),\incl)} \csmx[K,\zeta'] = (\freegraphmonad\csmx)(\eth(G), \incl)
\]
at $(\id, G, \zeta) \in \conngraph_{\fC}(\eth(G),\incl)$.
For $\freegraphmonad \gamma$ this follows because $\freegraphmonad \gamma$ is is given componentwise on 
\[
	\colim_{(f,K,\zeta') \in \conngraph_{\fC}(\eth(G),\incl)} \freegraphmonad\csmx[K,\zeta'] = (\freegraphmonad\freegraphmonad\csmx)(\eth(G), \incl),
\]
while for $\mu$ it follows from Remark~\ref{remark mu when G lacks vertices}.
But $\csmx[G,\zeta]$ is a point, so there is exactly one way for a function for factor through this structural map.
\end{proof}

\begin{proposition}
\label{prop modular operad associated to X}
The pair $(\csmx, \gamma)$ (from Definition~\ref{def csmx underlying object} and Definition~\ref{def of gamma restriction}) is an algebra over $\freegraphmonad_{\fC}$.
\end{proposition}
\begin{definition}
\label{def modular operad associated to X}
The pair $(\csmx, \gamma)$ is called the \emph{modular operad associated to $X$}.
\end{definition}
\begin{proof}[Proof of Proposition~\ref{prop modular operad associated to X}]
By Lemma~\ref{lemma gamma unit} we know that this pair satisfies the unit axiom.
Thus we must show that \eqref{heart} commutes, and it is enough to show that it commutes when restricted along the structural maps $(\freegraphmonad\csmx)[G,\zeta] \to (\freegraphmonad\freegraphmonad\csmx)(\eth(G), \incl)$.
Lemma~\ref{lemma no vertex algebra lemma} covers the case when $G$ has no vertices, so from now on we assume that $G$ has at least one vertex.
In particular, $G$ is an object of $\graphicalcat$.

The object $(\freegraphmonad\csmx)[G,\zeta]$ is (using the shorthand from \eqref{eq X zetavnbhd} appearing before Definition~\ref{def of gamma restriction})
\[
	\prod_{v\in G} \freegraphmonad\csmx(\zetavnbhd) = \prod_{v\in G} \colim_{\substack{(m_v, H_v, \zeta_v) \in \\ \conngraph_{\fC}(\zetavnbhd)}} \csmx[H_v, \zeta_v]; 
\]
we fix a collection $(m_v, H_v, \zeta_v) \in \conngraph_{\fC}(\zetavnbhd)$
and show that \eqref{heart} commutes when restricted to the natural map
\[
	\prod_{v\in G} \csmx[H_v, \zeta_v] \to (\freegraphmonad\csmx)[G,\zeta] \to (\freegraphmonad\freegraphmonad\csmx)(\eth(G), \incl).
\]

Composing with the arrow on the left of \eqref{heart} factors as
\begin{equation} \label{eq establishing dashed map} \begin{tikzcd}
\prod\limits_{v\in G} \csmx[H_v, \zeta_v] \rar \dar[dashed] & (\freegraphmonad\freegraphmonad\csmx)(\eth(G), \incl) \dar{\mu}  \\
 \csmx[G\{H_v\}, \tilde \zeta] \rar & (\freegraphmonad\csmx)(\eth(G), \incl) 
\end{tikzcd} \end{equation}
where $\tilde\zeta$ is from Construction~\ref{construction assembly}.
The dashed map comes as follows: each $\csmx[H_v, \zeta_v]$ is a subset of \[ X_{\segalcore[H_v]} = \hom(\segalcore[H_v], X) \xleftarrow{\cong} X_{H_v} \] 
while $\csmx[G\{H_v\}, \tilde \zeta]$ is a subset of 
\[
	X_{\segalcore[G\{H_v\}]} \xleftarrow{\cong} X_{G\{H_v\}}.
\]
We have a coequalizer diagram
\[ \begin{tikzcd}
\coprod\limits_{e \in E_i} \regc[\exceptionaledge] \rar[shift left, "\outeredge"] \rar[shift right, "\inneredge" swap] & \coprod\limits_{v\in V} \segalcore[H_v] \rar & \segalcore[G\{H_v\}],
\end{tikzcd} \]
and by applying $\hom(-,X)$, we have a monomorphism $X_{\segalcore[G\{H_v\}]} \hookrightarrow \prod_{v\in G} X_{\segalcore[H_v]}$.
Compatibility at the boundaries of the $H_v$ implies that the images of the monomorphisms
\[
	\csmx[G\{H_v\}, \tilde \zeta] \hookrightarrow X_{\segalcore[G\{H_v\}]} \hookrightarrow \prod_{v\in G} X_{\segalcore[H_v]}
\]
and $\prod \csmx[H_v, \zeta_v] \hookrightarrow \prod X_{\segalcore[H_v]}$ coincide.
This provides the dashed map in \eqref{eq establishing dashed map}.

The left bottom composite of \eqref{heart} is induced from the zigzag
\[
	\coprod_{v\in G} \segalcore[H_v] \to \segalcore[G\{H_v\}] \to \regc[G\{H_v\}] \leftarrow \regc[\medstar_G]
\]
(where $\medstar_G \to G\{H_v\}$ is the active map arising from $\eth(G) \cong \eth(G\{H_v\})$), that is, this zigzag induces
\[ \begin{tikzcd}
\prod\limits_{v\in G} \csmx[H_v, \zeta_v] \dar[hook] \rar[dashed, hook] & 
\csmx[G\{H_v\}, \tilde \zeta] \dar[hook] \arrow[rr,dashed] & & \csmx(\eth(G),\incl) \dar[hook] \\
\prod\limits_{v\in G} X_{\segalcore[H_v]} &  X_{\segalcore[G\{H_v\}]} \lar[hook'] & X_{G\{H_v\}} \lar["\cong"] \rar & X_{\medstar_G}
\end{tikzcd} \]

Let us turn to the top right of \eqref{heart}.
The top arrow arises from the diagram
\[ \begin{tikzcd}
\prod\limits_{v\in G} \csmx[H_v, \zeta_v] \dar[hook] \arrow[rr]  &  & \prod\limits_{v\in G} \csmx(\zetavnbhd) \dar[hook] \\
\prod\limits_{v\in G} X_{\segalcore[H_v]} &  \prod\limits_{v\in G} X_{H_v} \lar["\cong"] \rar & \prod\limits_{v\in G} X_{\medstar_v} 
\end{tikzcd} \]
where $\medstar_v \to H_v$ is the active map determined by $\eth(\medstar_v) = \nbhd(v)^\dagger \cong \dnbhd{v} \xrightarrow{m_v} \eth(H_v)$ for each $v$.
The vertical inclusion on the right factors through $X_{\segalcore[G]}$.
The arrow on the right of \eqref{heart} then comes from the diagram
\[ \begin{tikzcd}
\prod\limits_{v\in G} \csmx(\zetavnbhd)  \dar[hook] \arrow[rr, dashed] & & \csmx(\eth(G),\incl) \dar[hook] \\
X_{\segalcore[G]} & X_G \lar["\cong" swap] \rar & X_{\medstar_G}
\end{tikzcd} \]

We thus deduce commutativity of \eqref{heart} from commutativity of the following diagram of $\egc$-presheaves
\begin{equation} 
\label{reduced diagram}
\begin{tikzcd}
\segalcore[G] \rar \dar & \regc[G] \dar & \regc[\medstar_G] \lar \arrow[dl] \\
Z \rar & \regc[G\{H_v\}]  \\
\segalcore[G\{H_v\}] \arrow[ur] \uar
\end{tikzcd} \end{equation}
where the zig-zag on the left comes from the following pair of maps of coequalizers
\[ \begin{tikzcd}
\coprod\limits_{e \in E_i} \regc[\exceptionaledge] \rar[shift left, "\outeredge"] \rar[shift right, "\inneredge" swap] \dar["=" swap] & 
\coprod\limits_{v\in V} \regc[\medstar_v] \rar \dar & 
\segalcore[G]\dar 
\\
\coprod\limits_{e \in E_i} \regc[\exceptionaledge] \rar[shift left, "\tilde \outeredge"] \rar[shift right, "\tilde \inneredge" swap] & 
\coprod\limits_{v\in V} \regc[H_v] \rar & 
Z
\\
\coprod\limits_{e \in E_i} \regc[\exceptionaledge] \rar[shift left] \rar[shift right] \uar["="] & 
\coprod\limits_{v\in V} \segalcore[H_v] \rar  \uar& 
\segalcore[G\{H_v\}] \uar
\end{tikzcd} \]
But commutativity of \eqref{reduced diagram} is relatively straightforward. For instance, commutativity of the square follows from commutativity of 
\[ \begin{tikzcd}
\medstar_v \rar["\iota_v"] \dar & G \dar \\
H_v \rar & G\{H_v\}
\end{tikzcd} \] in $\egc$.
\end{proof}

\begin{lemma}
\label{lemma reverse direction of nerve theorem}
Let $Y$ be a $\graphicalcat$-presheaf satisfying the Segal condition and let $\csmx$ be the modular operad associated to $Y$ (Remark~\ref{remark graphcat vs egc segal and associated} and Definition~\ref{def modular operad associated to X}).
There is a canonical bijection \[\begin{tikzcd} Y_G\arrow[r, "f_{G}"] & (N\csmx)_G\end{tikzcd} \] for every $G\in\graphicalcat$.
The map $f$ is a morphism of $\pregraphcat$.
\end{lemma}
\begin{proof}
By definition of $\csmx$, there exist bijections 
\[\begin{tikzcd}
Y_{\exceptionaledge} \arrow[r, "\cong"] & (N\csmx)_{\exceptionaledge}=\csm(\rgc[\exceptionaledge], \csmx) \end{tikzcd} \] 
and 
\[\begin{tikzcd}
Y_{\medstar_m} \arrow[r, "\cong"] & (N\csmx)_{\medstar_m}=\csm(\rgc[\medstar_m], \csmx)=X_{\medstar_m}
\end{tikzcd} \] 
which are compatible with the embeddings ${\exceptionaledge} \to \medstar_m$.

For a graph $G$ with at least one internal edge, the map $f_{G}$ is given by the composition 
\[ 
\begin{tikzcd}
\pregraphcat(\rgc[G], Y) \arrow[d, "\cong"] \arrow[r,  "="]& Y_G\arrow[r, "f_G"] & (N\csmx)_{G}\arrow[r, "="]& \pregraphcat(\rgc[G], N\csmx) \dar["\cong"] \\
 \pregraphcat(\segalcore[G], Y) \arrow[rrr, "\lim f", "\cong" swap]&& & \pregraphcat(\segalcore[G], N\csmx) 
\end{tikzcd}
\] 
where the vertical arrows are the Segal maps of Definition~\ref{definition_Segal_condition}, and the two bijections in the top come from the Yoneda Lemma.
The bottom isomorphism follows from the first paragraph, and the Segal map for the nerve $N\csmx$ is a bijection by Lemma~\ref{Lemma_nerve_statisfies_Segal}.
\end{proof}

\begin{proof}[Proof of Theorem~\ref{nerve_theorem}]
We've already shown that the nerve functor is fully faithful in Proposition~\ref{nerve_fully_faithful}.
Satisfying the Segal condition is preserved by isomorphism, so one direction follows immediately from Lemma~\ref{Lemma_nerve_statisfies_Segal}.
The other direction is Lemma~\ref{lemma reverse direction of nerve theorem}.
\end{proof}

We already know (by \cite[\ONEsegaltosegal]{modular_paper_one}) that the category of Segal $\egc$-presheaves is equivalent to the category of Segal $\graphicalcat$-presheaves.
This latter category is equivalent to $\csm$, but we can say a bit more.
The functor $\graphicalcat \to \csm$ from \S\ref{sections: graphical csm} extends to $\egc \to \csm$ by sending a nodeless loop, with arc set $A$, to the initial object of $\csm_A$. Note that this extended functor is not injective on isomorphism classes of objects, as it was in Proposition~\ref{proposition functor graphical to csm}: the exceptional edge maps to the same modular operad as in Example~\ref{example_trivial_graphical_modular}.
\begin{remark}\label{remark egc nerve theorem}
The analogue of Theorem~\ref{nerve_theorem} holds for the functor $\egc \to \csm$.
Temporarily write $N'$ for the associated functor from $\csm$ to $\egc$-presheaves.
We have $N= \iota^* N'$. 
As any nodeless loop and the exceptional edge produce the same modular operad, we can conclude (using also Lemma~\ref{Lemma_nerve_statisfies_Segal}) that $N'(P)$ is Segal.
This also shows that $N' = \iota_* N$ by \cite[\ONEsegaltosegal]{modular_paper_one} (or more precisely, that the unit $N'(Q) \to \iota_* \iota^* N'(Q)$ is an isomorphism for each $Q$).
We have
\begin{align*}
\hom(N'(P), N'(Q)) &= \hom(N'(P), \iota_* \iota^* N'(Q)) \\
&= \hom(\iota^*N'(P), \iota^* N'(Q)) \\
&= \hom(N(P), N(Q)) \\
&= \csm(P,Q).
\end{align*}
so we see that $N'$ is fully-faithful.
Finally, the construction in \S\ref{section mod op segal presheaf} was already phrased in terms of Segal $\egc$-presheaves and the proof of Lemma~\ref{lemma reverse direction of nerve theorem} holds for $N'$.
\end{remark}

\begin{remark}
There is a category of colored cyclic operads (see \cite{DrummondColeHackney:DKHCO,Shulman:2CDCPMA}) $\mathbf{Cyc}$, which can be defined using monads as in Section~\ref{subsection CSM}, except only using simply-connected graphs with nonempty boundary.
Let $\graphicalcat_{\cycrm}$ denote the full subcategory of $\graphicalcat$ on the simply-connected graphs with nonempty boundary \cite[\ONEsecsimplconn]{modular_paper_one}.
There is a forgetful functor $\csm \to \mathbf{Cyc}$, and the composite
\[
	\graphicalcat_{\cycrm} \to \graphicalcat \to \csm \to \mathbf{Cyc}
\]
(where the middle functor is from Proposition~\ref{proposition functor graphical to csm}) is fully-faithful and injective on objects.
The reader should contrast this situation with Example~\ref{example about graphicalcat not being full} and \cite[Example 5.7]{hry_cyclic}.
We expect that $\graphicalcat_{\cycrm} \to \mathbf{Cyc}$ is thus amenable to the techniques of \cite{weber} and \cite{bmw}; in particular, the analogue of Theorem~\ref{nerve_theorem} may formally follow from Weber theory. 
\end{remark}

\section{The nerve theorem of Joyal and Kock}
\label{section: JK Nerve}

In Section~\ref{sections: graphical csm} we indicated how each graph determines a modular operad and that this constitutes a functor $J: \graphicalcat \to \csm$.
In fact, this factors as 
\[ \graphicalcat \xrightarrow{\iota} \jkgraphcat \xrightarrow{I} \csm,\]
where $\jkgraphcat$ (previously seen in Remark~\ref{refinement remark}) is the category of Feynman graphs of Joyal and Kock.
The latter functor in this composition appeared in \cite{JOYAL2011105}, though its existence shouldn't be surprising: 
Remark~\ref{remark embeddings} extends to \'etale maps, that is, every \'etale map $f: K \to G$ determines an element of $\modopgenned{G}(\eth(K),f|_{\eth(K)})$\footnote{Note, though, that $f|_{\eth(K)}$ need not be injective when $f$ is not an embedding. This implies that we cannot make the same choices we made in Remark~\ref{remark embeddings} for \'etale maps.}.

On page 112 of \cite{JOYAL2011105}, the following theorem is announced.
Details were promised in a forthcoming manuscript, which has not appeared in the intervening eight years.

\begin{theorem}[Joyal and Kock]
\label{jk nerve theorem}
The functor $I: \jkgraphcat \to \csm$ induces a fully faithful functor
\[
	\jknerve : \csm \to \prejkgraphcat
\]
where $\jknerve(P) = \csm(I(-), P)$.
The essential image of $\jknerve$ is characterized by the Segal condition.
\end{theorem}

The reader should note the similarities between this theorem and our Theorem~\ref{nerve_theorem}.
The purpose of the present section is to show how our nerve theorem implies that of Joyal and Kock.
This provides an independent proof of this theorem (whose original proof was never made public) using alternative techniques.
We would also like to point the reader to the thesis of Sophie Raynor \cite{Raynor:Thesis}, which takes a different approach to prove a related nerve theorem for modular operads.

The functor $\iota : \graphicalcat \to \jkgraphcat$ induces adjunctions
\begin{equation*}
\begin{tikzcd}[column sep=large]
		\prejkgraphcat \rar["\iota^*" description]  & \pregraphcat
	\lar[bend right=30, "\iota_!" swap] 
	\lar[bend left=30, "\iota_*"]  \end{tikzcd}
\end{equation*}
where $\iota_*$ (resp. $\iota_!$) is given by right (resp. left) Kan extension along $\iota^{\oprm} : \graphicalcat^{\oprm} \to \jkgraphcat^{\oprm}$.
The functor $\iota_!$ is the left adjoint of $\iota^*$, which in turn is the left adjoint of $\iota_*$.

The categories $\graphicalcat$ and $\jkgraphcat$ have the same set of objects, and by the Yoneda lemma the object $\iota_!\rgc[G]$ is the representable object $\jkgraphcat(-,G)$.
One can define the Segal condition via Segal cores exactly as in this paper, and then one would see that $\iota_! \segalcore[G] \to \iota_! \rgc[G]$ is the $\jkgraphcat$-analogue of the Segal core inclusion since $\iota_!$ is cocontinuous.
As the diagram
\begin{equation}
\label{adjunctions diagram}
\begin{tikzcd}
\pregraphcat(\rgc[G], \iota^* X) \rar \dar{\cong} & \pregraphcat(\segalcore[G], \iota^* X) \dar{\cong}\\
\prejkgraphcat(\iota_!\rgc[G], X) \rar &  \prejkgraphcat(\iota_!\segalcore[G], X)
\end{tikzcd} \end{equation}
is commutative for every $G$, we see that an object $X\in \prejkgraphcat$ satisfies the Segal condition if and only if $\iota^*X \in \pregraphcat$ satisfies the Segal condition.

\begin{lemma}\label{jk nerve is segal}
If $P$ is a modular operad, then $\jknerve(P) \in \prejkgraphcat$ satisfies the Segal condition.
\end{lemma}
\begin{proof}
Consider the diagram \eqref{adjunctions diagram} for $X = \jknerve(P)$ and for an arbitrary graph $G$.
Since $N(P) \cong \iota^* \jknerve(P)$, we know that the top map is an isomorphism by Lemma~\ref{Lemma_nerve_statisfies_Segal}, hence so is the bottom map.
\end{proof}

\begin{lemma}
\label{fully faithful and right kan extension}
The functor $\jknerve : \csm \to \prejkgraphcat$ is fully faithful.
Furthermore, there is a natural isomorphism of functors $\jknerve \cong \iota_* N$.
\end{lemma}
\begin{proof}
Consider the composition of functors
\[
	\graphicalcat \xrightarrow{\iota} \jkgraphcat \xrightarrow{I} \csm.
\]
The following two functors are fully faithful:
\begin{itemize}
	\item the functor $I : \jkgraphcat \to \csm$ (see \cite[\S 6]{JOYAL2011105}), and
	\item the functor $N(-) = \csm(I\iota,-)$ from $\csm$ to $\pregraphcat$ (Proposition~\ref{nerve_fully_faithful}).
\end{itemize}
Both statements of the lemma are then consequences of \cite[Proposition 1.1]{MR2366560}.
\end{proof}

We say that a graph is \emph{elementary} if it is isomorphic to either the exceptional edge $\exceptionaledge$ or to a star $\medstar_n$.

For the proof of the following lemma, it is convenient to utilize the pointwise description of right Kan extension (see, for instance, Theorem 1 of \cite[X.3]{maclane}).
Recall that if $Y\in \pregraphcat$, then
\[
	(\iota_*Y)_G = \lim_{\substack{G\downarrow \iota^{\oprm} \\ G \to H}} Y_H
\]
where $G\downarrow \iota^{\oprm}$ has objects $\jkgraphcat^{\oprm}(G,H) = \jkgraphcat(H,G)$ as $H$ varies, and morphisms from $G \to H$ to $G \to H'$ are those maps in $\graphicalcat^{\oprm}(H,H') = \graphicalcat(H',H)$ making the diagram 
\[ \begin{tikzcd}
G \rar \dar & H  \arrow[dl] \\
H' & 
\end{tikzcd} \]
commute in $\jkgraphcat^{\oprm}$.

\begin{lemma}
\label{lemma elem graphs}
The counit $\iota^*\iota_* \Rightarrow \id$ of the adjunction $\iota^* \dashv \iota_*$
is an isomorphism on each elementary graph.
In other words, for each presheaf $Y\in \pregraphcat$ and each elementary graph $K$, we have $(\iota_* Y)_K \cong Y_K$.
Furthermore, if $K$ and $K'$ are elementary graphs and $f \in \jkgraphcat(K,K')$, then $f\in \graphicalcat(K,K')$ and $(\iota_*Y)_f \cong Y_f$.
\end{lemma}
\begin{proof}
If $K$ is an elementary graph, then $\jkgraphcat(H,K) = \graphicalcat(H,K)$.
This implies that the object $\id_K^{\oprm}$ is initial in the category $K\downarrow \iota^{\oprm}$.
Thus the inclusion $\{ \id_K^{\oprm} \} \hookrightarrow K\downarrow \iota^{\oprm}$ induces an isomorphism
\[
	(\iota_*Y)_K = \lim_{K\downarrow \iota^{\oprm}} Y_H \to \lim_{\{ \id_K^{\oprm} \}} Y_K = Y_K.
\]
This proves the first statement.
The final sentence of the lemma follows immediately from naturality and the fact that $\jkgraphcat(K,K') = \graphicalcat(K,K')$.
\end{proof}

\begin{lemma}
\label{jk segal implies nerve}
If $X\in \prejkgraphcat$ satisfies the Segal condition, then $X\cong \jknerve(P)$ for some $P\in \csm$.
\end{lemma}
\begin{proof}
As mentioned above, $X$ satisfying the Segal condition is equivalent to $\iota^*X$ satisfying the Segal condition by the square \eqref{adjunctions diagram}.
Since $\iota^*X \in \pregraphcat$ is Segal, there exists a $P\in \csm$ and an isomorphism $\iota^* X \xrightarrow{\cong} N(P)$ by Theorem~\ref{nerve_theorem}.
We thus have an isomorphism 
\[ \iota_* \iota^* X \xrightarrow{\cong} \iota_* N(P)
\]
after right Kan extension; by Lemma~\ref{fully faithful and right kan extension} we know $\jknerve(P) \cong \iota_* N(P)$.

Write $f$ for the composite
\[
	X \to \iota_* \iota^* X \xrightarrow{\cong} \jknerve(P)
\]
where the first map is the unit of the adjunction $\iota^* \dashv \iota_*$.
We claim that for each elementary graph $K$, the first map $X_K \to (\iota_* \iota^* X)_K$ is an isomorphism. 
Indeed, this map is the first map in the composite
\[
	(\iota^*X)_K \to (\iota^*\iota_*\iota^*X)_K \to (\iota^*X)_K
\]
which is the identity function by one of the triangle identities for an adjunction (see Theorem 1(ii)(8) in \cite[XI.1, p.82]{maclane}).
We showed that the second map $(\iota^*\iota_*\iota^*X)_K \to (\iota^*X)_K$ was an isomorphism in Lemma~\ref{lemma elem graphs}, so the claim follows.

We now know that morphism $f$ of $\prejkgraphcat$ has the property that $f_K : X_K \to \jknerve(P)_K$ is an isomorphism for every elementary graph $K$.
Since both $X$ and $\jknerve(P)$ are Segal, this implies that $f$ is an isomorphism.
\end{proof}

\begin{proof}[Proof of Theorem~\ref{jk nerve theorem}]
By Lemma~\ref{fully faithful and right kan extension}, we know that $\jknerve$ is fully faithful.
As satisfying the Segal condition is invariant under isomorphism, we know by Lemma~\ref{jk nerve is segal} that every $X$ in the essential image of $\jknerve$ satisfies the Segal condition.
Lemma~\ref{jk segal implies nerve} provides the reverse containment.
\end{proof}

In this section we showed that Theorem~\ref{nerve_theorem} implies Theorem~\ref{jk nerve theorem}. 
This implication was mostly formal, relying on that fact that $\iota$ is a bijection on objects, the coincidence of the subcategories of elementary graphs, and fully faithfulness of $I: \jkgraphcat \to \csm$.
As there is no backwards version of Proposition 1.1 of \cite{MR2366560}, it seems unlikely that one can recover Theorem~\ref{nerve_theorem} from Theorem~\ref{jk nerve theorem}.

\bibliographystyle{amsalpha}
\bibliography{modular}
\end{document}